\documentclass[11pt]{amsart}
\usepackage[utf8]{inputenc}
\usepackage{hanamacrotemplate}
\usepackage{longtable}
\usepackage{marginnote}
\usepackage{todonotes}
\usepackage{mathtools}
\usepackage{soul}
\usepackage{adjustbox}
\usepackage{enumitem}
\usepackage[style=alphabetic,backend=bibtex,doi=false,isbn=false,url=false]{biblatex}
\renewbibmacro{in:}{}

\allowdisplaybreaks 
\edef\marginnotetextwidth{\the\textwidth}

\renewcommand{\l}{\overset}

\newcommand{\too}[1]{\l{#1}\to}
\newcommand{\hteq}{\simeq}
\renewcommand{\st}{\hspace{2pt} : \hspace{2pt}}
\newcommand{\isom}{\cong}
\newcommand{\til}{\widetilde}
\renewcommand{\Im}{\operatorname{Im}}

\definecolor{mgray}{RGB}{150,150,150}
\newcommand{\an}[1]{\left\langle {#1}\right\rangle}
\newcommand{\bto}{\leftarrow}
\newcommand{\btoo}[1]{\l{#1}\bto}
\newcommand{\djunion}[1]{\underset{#1}\amalg}
\newcommand{\ttoo}[1]{\l{#1}\longrightarrow}
\newcommand{\tensor}{\otimes}
\newcommand{\djunions}{\bigsqcup}
\renewcommand{\bar}[1]{{\overline{#1}}}
\newcommand{\Cyl}{\operatorname{Cyl}}
\newcommand{\M}{\mathcal{M}}
\newcommand{\OO}{\mathcal{O}}

\makeatletter
\newcommand{\switchmargin}{
\if@reversemargin
\normalmarginpar
\else
\reversemarginpar
\fi
}
\makeatother

\newcommand{\highlighteva}[1]{\ifmmode{\text{\sethlcolor{LightGray}\hl{$#1$}}}\else{\sethlcolor{PaleTurquoise}\hl{#1}}\fi}

\makeatother

\definecolor{llred}{RGB}{237,228,228}
\definecolor{llgray}{RGB}{230,230,230}
\definecolor{maroon}{RGB}{150,0,0}
\definecolor{orange}{RGB}{255,165,0}

\newcommand{\highlight}[1]{\ifmmode{\text{\sethlcolor{llgray}\hl{$#1$}}}\else{\sethlcolor{llred}\hl{#1}}\fi}

\newcommand{\argforcustom}{}
\theoremstyle{definition}
\newtheorem{helperforcustom}[equation]{\argforcustom}
\newtheorem*{helperforcustomstar}{\argforcustom}
\newenvironment{custom}[1]{\renewcommand{\argforcustom}{#1}\begin{helperforcustom}}{\end{helperforcustom}}

\hypersetup{pagebackref}

\newcommand{\cyl}{\otimes}

\newcommand{\fra}{\mathfrak{a}}
\newcommand{\frb}{\mathfrak{b}}

\newcommand{\Oiy}{\mathcal{O}_{\infty}}
\newcommand{\Mor}{\operatorname{Mor}}
\newcommand{\Kmod}{\mathcal{K}^{mod}}

\begin{document}
\title{A Toda bracket convergence theorem for multiplicative spectral sequences}
\author[Belmont]{Eva Belmont}
\author[Kong]{Hana Jia Kong}
\date{\today}
\maketitle

\begin{abstract}
Moss' theorem, which relates Massey products in the $E_r$-page of the classical Adams
spectral sequence to Toda brackets of homotopy groups, is one of the main tools
for calculating Adams differentials.
Working in an arbitrary symmetric monoidal stable simplicial model category,
we prove a general version of Moss' theorem which applies
to spectral sequences that arise from filtrations compatible with the monoidal structure. 
This involves the study of Massey products and Toda brackets in a non-strictly associative context.
The theorem has broad applications, e.g. to the computation of the motivic slice
spectral sequence and other colocalization towers.
\end{abstract}

\section{Introduction}
\label{sec:intro}

In homological algebra, Massey products \cite{massey} are operations that encode higher multiplicative structure on a differential graded algebra. Analogously in homotopy theory, Toda brackets (e.g. \cite{kochman-stable-book, Toda62}) encode higher multiplicative structure on homotopy classes of maps, for example on the homotopy groups of a ring spectrum.

In a nice enough spectral sequence computing homotopy groups of a ring spectrum, the multiplicative structure on $E_r$-pages of the spectral sequence agrees with the multiplicative structure on the homotopy groups. 
A theorem by Moss \cite[Theorem 1.2]{moss} states that in the Adams spectral sequence, there is a correspondence of higher multiplicative structures as well: given a Massey product of permanent cycles in the $E_r$-page such that the corresponding Toda bracket is defined, and assuming a ``crossing differentials'' hypothesis in the relevant degrees, the Massey product contains a permanent cycle that converges to an element of the Toda bracket.
This correspondence of Toda brackets and Massey products is used heavily in determining the differentials of Adams spectral sequences; for modern work with classical references, see \cite{stable-stems,IWX}.

We generalize Moss' theorem in two ways. While Moss' proof only applies to the Adams spectral sequence, our theorem holds for a wide range of spectral sequences,
with the motivic effective slice spectral sequence as a motivating example. 
Furthermore, while the Adams spectral sequence is strictly multiplicative, our
theorem applies in many cases where the filtered object associated to the
spectral sequence is only homotopy associative. This adds generality even when
the ring $R$ of interest is not strictly associative, as
strict associativity of $R$ does not
imply strict associativity of a given filtration $R_*$ of $R$.
We defer the discussion of hypotheses on the
spectral sequence to Section \ref{sec:intro-setup} and state the main theorems
next.

Although Theorem \ref{thm:main-intro} discusses Toda brackets
$\an{\alpha'',\alpha',\alpha}$
of classes $\alpha'',\alpha',\alpha \in \pi_*(R)$ for a ring spectrum $R$, the
generality used in the text also includes the case
$\alpha'',\alpha'\in \pi_*(R)$ and $\alpha\in \pi_*(M)$ for an $R$-module $M$.
\begin{thm}[Theorem \ref{thm:main}]
\label{thm:main-intro}
Let $R = R_0 \btoo{i} R_1\btoo{i} R_2 \bto \dots$ be a filtered object with a
multiplication $\mu$ satisfying the setup in Section \ref{sec:intro-setup}, and
consider the spectral sequence $\{ E_r(R) \}_r$ associated to this filtered object.
Fix $r\geq 2$.
Suppose $\fra,\fra',\fra''\in E_r(R)$ are permanent cycles converging to elements
$\alpha,\alpha',\alpha''\in \pi_*(R)$ such that $\fra'\fra$ and $\fra''\fra'$ are 0 in the
$E_r$-page, and $\mu(\alpha',\alpha)$ and $\mu(\alpha'',\alpha')$ are null. Suppose the spectral sequence is weakly
convergent (see Definition \ref{def:weakly-convergent}) and the crossing
differentials hypothesis (see Definition \ref{def:cross}) is satisfied in
degrees $(r-1, |\fra'\fra|)$ and $(r-1, |\fra''\fra'|)$. Then the Massey product
$\an{\fra'',\fra',\fra}\subseteq E_r(R)$ contains a permanent cycle that converges to an element of the Toda bracket $\an{\alpha'',\alpha',\alpha}\subseteq \pi_*(R)$.
\end{thm}

If $r\geq 2$, then $E_r(R)$ is the homology of the differential graded algebra
$(E_{r-1}(R), d_{r-1})$. In the $r=1$ case, we consider Toda brackets on $E_1(R)$ by regarding it as the homotopy groups of the associated graded objects $R_{\star,1}$ of the filtration on $R$.

\begin{thm} [Theorem \ref{thm:E1}] \label{thm:E1-intro}
Assume the setup of Theorem \ref{thm:main-intro} with $r=1$. Then the Toda bracket $\an{\fra'',\fra',\fra}\subseteq \pi_*(R_{*,1}) = E_1(R)$ contains a permanent cycle that converges to an element of the Toda bracket $\an{\alpha'',\alpha',\alpha}\subseteq \pi_*(R)$.
\end{thm}

These theorems hold for the motivic effective slice spectral sequence, as well as spectral sequences arising from more general colocalization towers (see Section \ref{sec:colocalization}).
In Section \ref{sec:R-motivic} we give examples of $\R$-motivic slice spectral
sequence computations using both theorems. 
Our theorems also hold for the Adams spectral sequence constructed in a general context (see Example \ref{example:Adams}).

\subsection{Toda brackets in a general setting}
This work is about multiplicative Toda brackets associated to a triple
$(X'',X',X)$ where $X'$ is an $X''$-algebra and $X$ is an $X'$-module. Closely related is the
\emph{composition Toda bracket}, associated to a diagram $Y\to
X''\to X'\to X$ where consecutive compositions are nullhomotopic.
The construction of both types of brackets in the stable homotopy category is classical; see \cite{kochman-uniqueness}
for references and a comparison.
The literature is less straightforward about the existence of Toda brackets in a symmetric monoidal stable model category, though such a construction is commonly assumed in practice.
Christensen and Frankland \cite{christensen-frankland} give several general
constructions of composition Toda brackets, along with a comparison.


Our version of the definition of a multiplicative Toda bracket (Definition \ref{def:toda}) has an extra complication because we do not require strict associativity of the system $(X'',X',X)$. In addition to providing extra generality, this is also natural to consider: even in the case that $(X'',X',X)$ is strictly associative, the filtered system $(X''_*,X'_*,X_*)$ may not be strictly associative.
Instead, in our motivating application (the motivic slice spectral sequence), the framework given by \cite{GRSO}
implies a structure on the filtered object that is more like an
$A_\infty$ structure; see Section \ref{sec:colocalization} for details.
Since the proof of the main theorem uses Toda bracket-like constructions on the filtered object (see \S\ref{sec:comparison-Er}), we consider Toda
brackets in a homotopy associative setting. No real complexity is then added in
passing to the setting where the original objects $(X'',X',X)$ are also assumed to
be only homotopy associative. Instead, we find this clarifies the setup as well
as being of independent interest.



\subsection{Summary of setup and assumptions} \label{sec:intro-setup}
For simplicity we state the case where $X''=X'=X$ is a ring object $R$; see
Section \ref{sec:setup} for the most general setup we consider, covering ring actions on modules as well. The assumptions below were designed to fit the known structure
of the motivic slice spectral sequence as described in \cite{GRSO},
specifically the precise associativity properties of the filtered object $R_*$.
Since \cite{GRSO} studies general colocalization towers,
this is also a useful framework for other spectral sequences arising from such towers.

Let $(\sC, \tensor, S)$ be a symmetric monoidal stable simplicial model category with cofibrant unit $S$ and consider an $A_3$-ring $R$; this is a homotopy associative ring object $(R,\mu)$ with a choice of homotopy $\mu(\mu(-,-),-)\hteq \mu(-,\mu(-,-))$.
Suppose there is a filtration $R=R_0 \btoo{i} R_1\btoo{i} R_2\btoo{i} \dots$ and let $R_{s,r} = \hocolim(R_{s+r}\to R_s)$. Consider
the spectral sequence 
$$E_1(R)=\pi_*(R_{f,1})\implies \pi_*(R)$$
where $\pi_*(-)$ denotes $\pi_*(\Map_\sC(S,-))$.
Assume that there are pairings $\mu_{s,t}:R_s\tensor R_t \to R_{s+t}$ and $\mu_{s,t,1}:R_{s,1}\tensor R_{t,1}\to R_{s+t,1}$ for $s,t\geq 0$. We make the following assumptions:
\begin{enumerate} 
\item The pairings make $\iota: R_*\to R$ a map of $A_3$-rings, and the natural map $p:R_*\to R_{*,1}$ commutes with the pairings. That is, there are strictly commutative diagrams
$$ \xymatrix{
R_s\tensor R_t\ar[r]^-{\mu_{s,t}}\ar[d]_\iota & R_{s+t}\ar[d]^\iota
\\R\tensor R\ar[r]^-\mu & R
}\hspace{30pt}
\xymatrix{
R_s\tensor R_t\ar[r]^-{\mu_{s,t}}\ar[d]_p & R_{s+t}\ar[d]^p
\\R_{s,1}\tensor R_{t,1}\ar[r]^-{\mu_{s,t,1}}\ar[r] & R_{s+t,1}
} $$
and the associativity homotopies for $R_*$ and $R$ are strictly compatible.
\item A condition that is less strict than asking
$$ \xymatrix{
R_s\tensor R_t\ar[r]^-{\mu_{s,t}}\ar[d]_i & R_{s+t}\ar[d]^i
\\R_s\tensor R_{t-1}\ar[r]^-{\mu_{s,t-1}} & R_{t-1}
}\hspace{30pt}
\xymatrix{
R_s\tensor R_t\ar[r]^-{\mu_{s,t}}\ar[d]_i & R_{s+t}\ar[d]^i
\\R_{s-1}\tensor R_t\ar[r]^-{\mu_{s-1,t}} & R_{t-1}
} $$
to commute, but stricter than asking these diagrams to commute up to homotopy; see Assumption \ref{ass:X}\eqref{item:mu_r},\eqref{item:H_1}.
\item The pairing $\mu_{*,1}$ induces a pairing of spectral sequence $E_r$-pages.
\end{enumerate}
Some other results on the multiplicativity of spectral sequences (e.g.
\cite{davis-snaith,hedenlund-rognes}) are formulated in terms of
Cartan-Eilenberg systems, which makes assumptions only on the system of
quotients $\{R_{s,r}\}$. Our setting also includes assumptions on the
multiplicativity of the filtration $R_*$ itself. Our setting is more general than the setting of monoids in
filtered spectra; see Remark \ref{rmk:filtered}.

The idea behind the proof of Moss' theorem (Theorems \ref{thm:main-intro},
\ref{thm:E1-intro}) is to lift $\fra,\fra',\fra''\in E_r(R)$ to $\alpha_*,\alpha'_*,\alpha''_*$ in $\pi_*(R_*)$ and show that there is a comparison of brackets
$$ \an{\fra'',\fra',\fra}\btoo{p} \an{\alpha''_*,i^r\alpha'_*, \alpha_*}\too{\iota} \an{\alpha'',\alpha',\alpha}, $$
using the crossing differentials hypothesis (see Definition \ref{def:cross}, Assumption \ref{ass:crossing}) to ensure that the middle term is defined.
In practice, the crossing differentials hypothesis is satisfied in many cases.

This strategy essentially works for Theorem \ref{thm:E1-intro}, but for Theorem \ref{thm:main-intro}, 
we need to replace the element of $\an{\alpha''_*,i^r\alpha'_*,\alpha_*}$ with a more complicated element, corresponding to the extra technical conditions in Assumption \ref{ass:X}\eqref{item:mu_r},\eqref{item:H_1}.

\subsection{Outline} In Section \ref{sec:sseq-setup}, we introduce notation for the spectral sequence associated to a tower, including a review of the crossing differentials criterion. In Section \ref{sec:toda} we give notation for the setup of the main theorem, namely a pairing of spectral sequences satisfying certain properties along with certain elements that will form the relevant Toda brackets and Massey products. We define 3-fold Toda brackets and Massey products in our general setting. In Section \ref{sec:comparison} we prove our versions of Moss' theorem (Theorems \ref{thm:main} and \ref{thm:E1}). In Section \ref{sec:examples} we show that the assumptions in our setup apply to colocalization towers, and
give some example applications of these theorems to specific computations in the $\R$-motivic slice spectral sequence.

The reader who wishes to jump to the precise statements of the theorems should 
read Assumption \ref{ass:X}, the notation at the beginning of Section \ref{sec:assumptions2}, and Assumption \ref{ass:crossing}, and then skip to Section \ref{sec:comparison}.

\subsection{Acknowledgements}
We would like to thank Dan Isaksen for helpful conversations and comments, and Peter May for helpful comments on our manuscript.
We would also like to thank John Rognes for finding an error in an earlier
version of the manuscript.
The first author was supported by the National Science Foundation under Grant No. DMS-2204357.
The second author was supported by the National Science Foundation under Grant No. DMS-1926686.

\section{The spectral sequence associated to a tower}\label{sec:sseq-setup}
\subsection{The spectral sequence for a tower of homotopy cofiber sequences}
We work in a stable symmetric monoidal simplicial model category
$(\sC,\tensor,S)$ with cofibrant unit.
Suppose we have a tower
\begin{equation}\label{eq:tower} \xymatrix{
X=X_0\ar[d]_-p & X_1\ar[l]_-i & \dots\ar[l] & X_f\ar[l]_-i\ar[d]_-p &
X_{f+1}\ar[l]_-i & \dots\ar[l]_-i
\\X_{0,1}\ar@{.>}[ru]_-\kappa &&& X_{f,1}{}\ar@{.>}[ru]_-\kappa & 
}\end{equation}
in $\sC$, where $X_{f,r}$ for $r\geq 1$ is defined to be the homotopy cofiber
$$ X_{f,r} := X_f\djunion{X_{f+r}} CX_{X_{f+r}}. $$
We abuse notation so $i$ refers to any of the maps $X_{f+1}\to X_f$ or any of their iterates.

Applying $\pi_*(-)$, we have a spectral sequence
$$ E_1^{n,f}(X)=\pi_n(X_f/X_{f+1})\Rightarrow \pi_n(X). $$
We will use the bigrading convention $(n,f)$ where $n$ is the stem and $f$ is
the filtration.
It will be convenient for us to have a way to view cycles and
boundaries as subsets of $E_1$:
for $r\geq 2$, let
\begin{align*}
Z_r^{*,f} & = \{ a\in \pi_*(X_{f,1}) \st \kappa(a)\in \pi_{*-1}(X_{f+1}) \text{ lifts to }\pi_{*-1}(X_{f+r}) \} \text{, and}
\\B_r^{*,f} & =\begin{cases}
 	\{ a\in \pi_*(X_{f,1})\st a\text{ lifts to some }\alpha\in
\ker(\pi_*(X_f)\to \pi_*(X_{f-r+1})) \} &  f\geq r- 1,\\
0  & f<r-1.
 \end{cases}
\end{align*}
Say that an element of $Z_r^{*,*}$ is a \emph{$d_{r-1}$-cycle}, and say that an
element of $B_r^{*,*}$ is a \emph{$d_{r-1}$-boundary}. Note that there is an inclusion $B_r^{*,f}\subseteq Z_r^{*,f}$ since $\kappa(a)=0$ for $a\in B_r^{*,f}$.
Let
$$ E_r^{*,*} = Z_r^{*,*}/B_r^{*,*}. $$
Given $x\in Z_r^{*,*}$, let $[x]_r$ denote its image in $E_r^{*,*}$.
Then there is a map $d_r:E_r^{n,f}\to
E_r^{n-1,f+r}$ defined as follows: given $a\in Z_r^{n,f}$, let $x\in
\pi_{n-1}(X_{f+r})$ be a lift of $\kappa(a)$, and define $d_r(a)=px$. (To see
that this is well-defined, observe that the difference of two lifts is in
$\ker(\pi_*(X_{f+r})\to \pi_*(X_{f+1}))$.)

\begin{rmk}\label{rmk:EZR}
This is not the usual definition of cycles and boundaries, but we claim that the
$E_r$-page as defined above is isomorphic to the usual $E_r$-page, defined as a
subquotient of $E_1$.
A cycle in the usual sense is an equivalence class of
elements $a\in E_1(X)$ in the kernel of $d_i$ for all $i\leq r-1$, modulo
$d_i$-boundaries for $i< r-1$. Our $Z_r^{*,*}$ consists of all such elements $a$,
since $\kappa(a)$ lifts to $\pi_{*-1}(X_{f+r})$ if and only if its image in
$\pi_{*-1}(X_{f+r-1,1})$ is zero. Our $B_r^{*,*}$ consists of all
$d_i$-boundaries for $i\leq r-1$:
a differential $d_i(b)=a$ means that there
is a class $\til{a}\in \pi_*(X_f)$ that projects to $a\in \pi_*(X_{f,1})$ and whose image in
$\pi_*(X_{f-i+1})$ equals $\kappa(b)$; hence the image of $\til{a}$ in
$\pi_*(X_{f-i})$ is zero.
\end{rmk}

\begin{definition}[see {\cite[Definition 3.8]{mccleary}}]\label{def:weakly-convergent}
Let
$$ R_f = \bigcap_{r\geq 0} \Im(\pi_*(X_{f+r})\to \pi_*(X_f)).  $$
Say that a spectral sequence $\{ E_*(X) \}$ is \emph{weakly convergent} if the
natural map $R_{f+1}\to R_f$ is an injection for all $f$.
\end{definition}

\subsection{The crossing differentials condition}\label{sec:crossing-diff}
If $a\in E_r^{n,f}$ is a $d_r$-boundary, then by definition it has a lift to $\pi_*(X_f)$
(i.e., $\alpha\in \pi_n(X_f)$ with $p\alpha = a$) that is in $\ker(\pi_*(X_f)\to
\pi_*(X_{f-r}))$. Proposition \ref{prop:crossing-diff} says that \emph{every}
lift is in the kernel, as long as we assume a criterion called the crossing differentials hypothesis.

\begin{definition}
\label{def:cross}
Say that a spectral sequence $\{E_*(X)\}$ \emph{satisfies the crossing
differentials hypothesis for $(r,n,f)$} if
\begin{equation} \label{eq:crossing} \text{every element in }E^{n+1,m}_{f-m+1}(X)\text{ is a permanent
cycle for } 0\leq m\leq f-r-1. \end{equation}
If an element $y$ is in stem $n$ and filtration $f$, write $(r,|y|):= (r,n,f)$.
\end{definition}
In the context of a differential $d_r(x)=y$ for $x\in E_r^{n+1,f-r}(X)$,
$y\in E_r^{n,f}(X)$, the crossing differentials hypothesis is satisfied in degree
$(r,|y|)$ if there are no differentials
from stem $n+1$ to stem $n$ with source filtration $<f-r$ and target filtration $>f$.
$$
\begin{tikzpicture}
\filldraw[black] (-1,1) circle (1pt) [anchor=east] node {$y$};
\draw[->] (0,0) -- (-1,1);
\draw[->,dashed] (0,-0.5) -- (-1,1.5);
\filldraw[black] (0,0) circle (1pt)[anchor=west] node {$x$};
\filldraw[black] (-0.8,1.3) circle (0pt)[anchor=west] node {\tiny a crossing differential};
\end{tikzpicture}
$$

\begin{prop}[{\cite[Proposition 6.3]{moss}}]\label{prop:crossing-diff}
Assume that $\{ E_*(X) \}$ is weakly convergent. 
Suppose $a\in \pi_*(X_{f,1})$ is a $d_r$-boundary, and suppose
$\alpha\in \pi_*(X_f)$ is any lift of $a$ such that $\alpha$ is in $\ker(\pi_*(X_f)\to \pi_*(X))$.
If the crossing differentials hypothesis is satisfied in degree $(r,|a|)$, then
$\alpha$ is in $\ker(\pi_*(X_f)\to \pi_*(X_{f-r}))$; i.e., the composite
$S\xrightarrow{\alpha} X_f\xrightarrow{i} X_{f-r}$ is nullhomotopic.
\end{prop}

\begin{proof}
	Let $b\in\pi_{*+1}(X_{f-r,1})$ be the source of the differential that hits $a$. Thus we have a lift $\beta\in \pi_*(X_f)$ of $\kappa(b)$.
	If $\beta = \alpha$, the result follows from the long exact sequence. Otherwise, denote the nonzero element $\alpha-\beta \in \pi_*(X_f)$ by $\gamma$.
	Since $\alpha\in \ker(\pi_*(X_f)\to \pi_*(X))$, there is a minimal $k$ for which $\alpha\in\ker(\pi_*(X_f)\to \pi_*(X_{f-k}))$.
	Assume for the sake of contradiction that $k>r$. Since $\beta\in \ker(\pi_*(X_f)\to \pi_*(X_{f-r}))$ we have $\gamma\in \ker(\pi_*(X_f)\to \pi_*(X_{f-k}))$ and $\gamma\notin \ker(\pi_*(X_f)\too{i} \pi_*(X_{f-k+1}))$. So there exists $c\in \pi_*(X_{f-k,1})$ such that $\kappa(c) = i\circ \gamma$.

	By the weak convergence criterion (Definition \ref{def:weakly-convergent}), there is some $n$ such that $\gamma$ lifts to $\gamma'\in\pi_*(X_{f+n})$ but not to $\pi_*(X_{f+n+1})$.
	Since $\gamma = \alpha-\beta\in \ker(\pi_*(X_f)\to \pi_*(X_{f,1}))$, we have $n\geq 1$. Then $0\neq p\gamma'\in \pi_*(X_{f+n,1})$, and
	we conclude that $d_{k+n}(c)=p\gamma'$. This contradicts the crossing differentials hypothesis, concluding the proof, as long as the target is nonzero.

	If $p\gamma'$ were a $d_i$-boundary for $i\leq k+n$, then it would have a lift $\gamma''\in \ker(\pi_*(X_{f+n})\to \pi_*(X_{f-k}))$. Then $\gamma''-\gamma'$ lifts to $\delta \in \pi_*(X_{f+n+1})$ whose image in $\pi_*(X_{f-k})$ equals $\gamma$, contradicting the definition of $n$.
	\end{proof}

Proposition \ref{prop:crossing-diff} is used in Lemma
\ref{lem:compositions-null}, where we have a specified lift (given as a product) to $\pi_*(X_f)$ of a certain class in $\pi_*(X_{f,1})$; Proposition \ref{prop:crossing-diff} guarantees that the chosen lift becomes zero in $\pi_*(X_{f-r})$, whereas a priori we only know that there exists a lift with this property.

\subsection{Gluing nullhomotopies}
We assume functorial cones, i.e. if $f:X\to Y$ is a map then there is a natural
map $f:CX\to CY$. We assume that $S$ is cofibrant, which implies a familiar
model for $CS$, but not any particular model for $CX$ for general $X$.

\begin{custom}{Notation}
Since $\sC$ is a simplicial model category, a model for $CS$ is $(S \x
\Delta^1)/(S \x \{ 0 \})$, which we write as $(S \x I)/(S \x \{ 0 \})$ (where
$I$ represents the unit interval).
Given a homotopy $F:S\x I \to X$, let $\bar{F}$ denote the homotopy with
reversed orientation $\bar{F}(x,t)=F(x,1-t)$.
We use the convention that $\Sigma X$ denotes the Kan suspension $CX/X$. Let
$$ S^\diamond = (S \x I)/(S \x \{ 0 \}) \djunion{S} (S \x I)/(S \x \{ 1 \}). $$
This is the mapping cone of the inclusion $S\to CS$, and $S^\diamond \hteq
\Sigma S$.
\end{custom}

In a model category, a homotopy between $f,g:X\to Y$ is a map $\Cyl(X)\to Y$
satisfying an appropriate diagram, and if $X$ is cofibrant in a simplicial model
category we may take $X\tensor I$ as the cylinder object $\Cyl(X)$ (see e.g.
\cite[Lemma II.3.5]{goerss-jardine}). We will refer to a map $X\tensor I \to Y$
as a homotopy, which is justified by the fact that we are ultimately concerned
with the composition
$$ S\tensor I \to X\tensor I \to Y $$
which is a homotopy, due the assumption that $S$ is cofibrant.

\begin{definition} \label{def:d()}
Suppose we have maps $F_0,F_n: CS\to X$ and $F_i : S \x I\to X$ for $1\leq i \leq n-1$ such that
$$ F_i|_{S \x \{ 0 \}} = F_{i-1}|_{S \x \{ 1 \}}$$
for $1\leq i \leq n-1$ and $F_n|_{S\x \{ 1 \}} = F_{n-1}|_{S\x \{ 1 \}}$. Define $d(F_0, \dots, F_n):\Sigma S \to X$ to be
$$ \Sigma S \isom (S \x I)/(S \x \{ 0 \}) \djunions_{S}
(S \x I)\djunions_{S} \dots \djunions_{S} (S \x I)/(S \x \{ 1
\})\ttoo{(F_0,\dots,\bar{F_n})} X.  $$
\end{definition}

We present the following figure to illustrate the definition.

\begin{figure}[H]
\includegraphics[width=200pt]{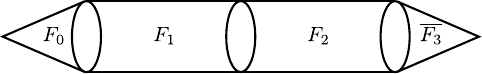}
\end{figure}

\begin{example}
Given nullhomotopies $F,G:CS\to X$, the induced map out of the pushout
$S^\diamond$ in the following diagram is $d(F,G)$.
$$ \xymatrix{
S\ar[r]\ar[d] & CS\ar[d]\ar@/^1pc/[rd]^-F
\\CS\ar[r]\ar@/_1pc/_-G[rr] & S^\diamond\ar@{.>}[r] & X
}$$
\end{example}

\begin{lemma}\label{lem:difference}
Given the setup of Definition \ref{def:d()},
$$ d(F_0,\dots,F_n) = -d(F_n,\bar{F_{n-1}},\dots,\bar{F_1}, F_0). $$
\end{lemma}
The reason that the endpoints do not need their orientation reversed is the
convention that $d(F_0,\dots,F_n)$ refers to the map with components
$(F_0,\dots,\bar{F_n})$.
\begin{proof}
To simplify the notation, we will show that
$S\x [0,2]\too{(F,\bar{F})} X$ is nullhomotopic
for any $F:CS\x I \to X$. The full statement is proved the same way.
Define $H:S\x [0,2]\x [0,1]\to X$ by
$$ H(x,s,t) = \begin{cases} F(x,0) & \text{ if } s \leq t
\\F(x,s-t) & \text{ if } t < s \leq 1
\\F(x,2-s-t) & \text{ if } 1 < s < 2-t
\\F(x,0) & \text{ if } s \geq 2-t.
\end{cases} $$
This is a homotopy from $(F,\bar{F})$ to the constant map at $F(x,0)$.
\end{proof}

\begin{lemma} \label{lem:colimit}
Given $F_0,\dots,F_n$ as in Definition \ref{def:d()} and $g:X\to Y$, we have
$$ g\circ d(F_0,\dots,F_n) = d(g\circ F_0,\dots,g\circ F_n). $$
\end{lemma}
\begin{proof}
This is by definition of the $d(-)$ construction.
\end{proof}

Using this notation we obtain an expression for the source of a $d_r$ differential
with given target.
\begin{lemma} \label{lem:source}
Suppose $\fra\in \pi_*(X_{f,1})$ is a $d_r$-boundary and $\alpha\in
\pi_*(X_f)$ is a lift in $\ker(\pi_*(X_f)\to \pi_*(X_{f-r}))$.
Let $F:CS\to X_{f-r}$ be any nullhomotopy making the diagram commute
\begin{equation}\label{eq:d_r(beta)} \xymatrix{
S\ar[r]^-\alpha\ar[d] & X_f\ar[r]^-i\ar[d] & X_{f-r+1}\ar[d]\ar@/^1.5pc/[dd]^j
\\CS\ar[r]^-F\ar[d] & X_{f-r}\ar@{=}[r] & X_{f-r}\ar[d]_p
\\S^\diamond \ar@{.>}[rr]^-\beta &  & X_{f-r,1}
}\end{equation}
and let $\beta = d(pF, jC\alpha)$ where $j:CX_f \to CX_{f-r+1}\to X_{f-r,1} = X_{f-r}\djunion{X_{f-r+1}}CX_{f-r+1}$ is the natural map. Then
$$ d_r([\beta]_r) = [\fra]_r. $$
\end{lemma}
Recall that we abbreviate iterates of $i$ by just $i$.
\begin{proof}
From the diagram
$$ \xymatrix{
X_{f-r}\ar[d] & X_{f-r+1}\ar[l] & \ar[l]_-i\dots & X_f\ar[d]\ar[l]_-i
\\X_{f-r,1}\ar@{.>}[ru]_-{\kappa} & & & X_{f,1}
}$$
we see that we need to check that $\kappa\circ \beta = i\circ \alpha$ as elements of
$\pi_*(X_{f-r+1})$.
By construction, $\beta$ is the induced map from the homotopy cofiber of $S\to CS$ to the homotopy cofiber of $X_{f-r+1}\to X_{f-r}$.
So the two vertical cofiber sequences in \eqref{eq:d_r(beta)} can be extended as follows.
$$ \xymatrix{
S^\diamond \ar@{.>}[rr]^-\beta \ar[d]_-\hteq &  &
X_{f-r,1} = X_{f-r}\djunion{X_{f-r+1}} CX_{f-r+1}\hspace{-100pt}
\ar[d]^-{\kappa}
\\\Sigma S\ar[rr]^-{i\circ \alpha} &  & \Sigma X_{f-r+1}
}$$
\end{proof}

\section{Toda bracket and Massey product setup}
\label{sec:toda}
Let $(\sC, \tensor, S)$ be a symmetric monoidal stable simplicial model category
with cofibrant unit.
\subsection{Pairings}
Our eventual goal is to discuss higher multiplicative structures; we start by giving notation for the ordinary multiplication of homotopy classes of maps.

\begin{definition}\label{def:mult}
Given a pairing $\mu:X'\tensor X \to X$ and classes $\alpha:\Sigma^nS \to X$ and $\alpha':\Sigma^mS\to X'$, the product $\mu(\alpha',\alpha)$ is defined as the composition $\Sigma^{m+n}S\isom \Sigma^{m}S\tensor \Sigma^{n}S\too{\alpha'\tensor \alpha} X'\tensor X \too{\mu}X$, where the first isomorphism is the unit isomorphism from the symmetric monoidal structure, which commutes with suspension.
\end{definition}
\begin{rmk}\label{rmk:signs}
In the rest of the paper, we will make the simplifying assumption that all maps
have degree zero, i.e. have source $S$. The only change required for nonzero degrees concerns the
insertion of minus signs. These are the same as in the classical case, which is
well known; see for example \cite[\S1.4]{green} for the signs in the definition
of Massey products on a differential graded algebra.
\end{rmk}

\begin{rmk}\label{rmk:X'-to-X-mult}
Given a pairing $\mu:X'\tensor X\to X$ and classes $\alpha:S\to X$ and $\alpha':S\to X'$, the pairing $S\too{\alpha'\tensor \alpha} X'\tensor X \too{\mu} X$ can 
be factored using the unit axiom in two different ways:
\begin{align*}
S\too{\alpha'} \underbracket[1pt]{X'\tensor S \to X' \to X'\tensor S}_{\Id} \too{\alpha} X'\tensor X \too{\mu} X
\\S\too{\alpha} \underbracket[1pt]{S\tensor X \to X \to S\tensor X}_{\Id}\too{\alpha'} X'\tensor X \too{\mu} X
\end{align*}
and so we can regard $\mu(\alpha',\alpha)$ as a map $S\too{\alpha'}
X'\too{\alpha} X$, or equivalently $S\too{\alpha}X\too{\alpha'}X$. (This is not a commutativity statement.)

If $F:CS\to X'$ and $G: CS\to X$ are nullhomotopies, we define $\mu(F,\alpha)$ and $\mu(\alpha',G)$ as the compositions
\begin{align*}
\mu(F,\alpha): CS\too{F} X' \too{\alpha} X
\\\mu(\alpha',G): CS \too{G} X \too{\alpha'} X.
\end{align*}
\end{rmk}

Given a pairing $E_r(X')\tensor E_r(X)\to E_r(X)$, we define the product of classes $\fra'\in E_r(X')$ and $\fra\in E_r(X)$ to be the image along the pairing.

\subsection{$A_3$-structures on a triple}
The motivation is to study ring objects in $\sC$ with homotopy associative
filtrations. To package
the necessary compatibilities and introduce the generality of module actions in
addition to ring multiplications, we use the language of operads.
More precisely, let $\mathcal{K}$ denote the Stasheff associahedron operad and
recall that an $A_n$-algebra is a partial action of $\mathcal{K}$ ignoring
the structure maps involving $\mathcal{K}(m)$ for $m> n$ (see e.g.
\cite{robinson-Ainfty}). 
Since $\mathcal{K}(1)=\mathcal{K}(2)=*$ and $\mathcal{K}(3)=I$,
an $A_3$-algebra in $\sC$ is simply a homotopy associative ring with a choice of
homotopy $\mu(\mu(-,-),-)\hteq \mu(-,\mu(-,-))$.
Given an $A_3$-ring $R$, the notion of an $A_3$-module $M$ over $R$ can be
defined in terms of colored operads (see \cite[\S A.1.1]{GRSO}), but
this boils down to the data of a map $\mu:R\tensor M\to M$ and the choice of a
homotopy between $(R\tensor R)\tensor M\to M$ and $R\tensor(R\tensor M)\to M$.
We work in somewhat more generality.

\begin{definition}\label{def:A3-structure}
Given objects $X''$, $X'$, and $X$ in $\sC$, an $A_3$-structure on the
triple $(X'',X',X)$ consists of the following pieces of structure:
\begin{itemize} 
\item a map $\mu: X'\tensor X \to X$,
\item a map $\mu':X''\tensor X' \to X'$,
\item a map $\mu'':X''\tensor X \to X$, and
\item a map $h:X''\tensor X'\tensor X\tensor I\to X$ whose restrictions to $X''\tensor X'\tensor X\tensor \{ 0 \}$ and $X''\tensor X'\tensor X\tensor \{ 1 \}$ are
$\mu(\mu'(-,-),-)$ and $\mu''(-,\mu(-,-))$, respectively.
\end{itemize}
A map $(X'',X',X)\to (Y'',Y',Y)$ of $A_3$-triples consists of maps $X''\to Y''$,
$X'\to Y'$, and $X\to Y$ such that the obvious diagrams commute strictly.
\end{definition}

To avoid notational clutter, we will use the same symbol $\mu$ to denote $\mu$, $\mu'$
and $\mu''$. We also write $(\mu,h)$ to abbreviate the full
tuple of structure data $(\mu,\mu',\mu'',h)$.

\begin{rmk}
In a model category, a homotopy between two maps $X''\tensor X'\tensor X\to X$
is a map out of $\Cyl(X''\tensor X'\tensor X)$, but this is not a problem for us
as we will eventually precompose with a map out of the unit object $S$, and
cofibrancy of the unit implies that $\Cyl(S) = S\tensor I$.
\end{rmk}


\begin{example}
In the case $X=X'=X''$, an $A_3$-ring structure on $X$ gives rise to an $A_3$-structure on
$(X,X,X)$, and a map $X\to Y$ of $A_3$-rings gives rise to a morphism of $A_3$-structures $(X,X,X)\to (Y,Y,Y)$. More generally, if $X''\to X'$ is a map of $A_3$-rings and $X$
is an $A_3$-module over $X'$, then this endows $(X'',X',X)$ with the structure
of an $A_3$-triple.
\end{example}

\subsection{Toda brackets}
We define the multiplicative Toda bracket $\an{{\alpha},{\alpha}',{\alpha}''}$ following \cite[Definitions 2.2.1, 2.2.2]{kochman-stable-book},
modified to allow for the possibility that the multiplications are not strictly
associative.
Let $CX$ denote the homotopy cofiber of $X\too{\Id} X$.

\begin{definition}\label{def:toda}
Suppose $(\mu,h)$ is an $A_3$-structure on the triple $(X'',X',X)$ and
let $\alpha'':S\to X''$, $\alpha':S\to X'$, and $\alpha:S\to X$ be maps such that $\mu(\alpha',\alpha)$ and $\mu(\alpha'',\alpha')$ are nullhomotopic. Let $H$ be a homotopy from $\mu(\mu(\alpha'',\alpha'),\alpha)$ to $\mu(\alpha'',\mu(\alpha',\alpha))$ defined by
$$ \xymatrix@1@C=40pt{ H: S\cyl I\ \ \  \ar[r]^-{\alpha''\tensor \alpha'\tensor \alpha} &  \ \ \ X''\tensor X'\tensor X\tensor I\ \ \ \ar[r]^-h &  \ \ \ X.} $$
	The \emph{Toda bracket} $\an{\alpha'',\alpha',\alpha}$
	is the set of homotopy classes of maps
	$$ d(\mu(F,\alpha), H, \mu(\alpha'',G)) $$
	where $F:CS \to X'$ is any nullhomotopy of $\mu(\alpha'',\alpha')$ and
	$G:CS\to X$ is any nullhomotopy of $\mu(\alpha',\alpha)$.
(See Remark \ref{rmk:X'-to-X-mult} for multiplication of nullhomotopies.)
\end{definition}

It is helpful to represent this using the following diagram.
$$ \xymatrix@R=10pt{
CS\ar@/^1pc/[rd]^-F
\\S\cyl \{ 0 \}\ar@{^(->}[u]\ar@{_(->}[d]\ar[r]^-{\mu(\alpha'',\alpha')} &
X'\ar[r]^-{1\tensor \alpha} & X'\tensor X\ar[r]^-\mu &  X\ar@{=}[d]
\\S\cyl I\ar[r]^-{\alpha''\tensor \alpha'\tensor \alpha} & (X''\tensor X'\tensor X)\cyl I \ar[rr]^-{h} &  & X\ar@{=}[d]
\\S\cyl \{ 1 \}\ar@{^(->}[u]\ar@{_(->}[d]\ar[r]^-{\mu(\alpha',\alpha)} & X\ar[r]^-{\alpha''\tensor 1} & X''\tensor X\ar[r]^-{\mu} & X
\\CS\ar@/_1pc/[ru]_G
} $$

\begin{rmk}
If $(X'',X',X)$ is strictly associative, we take $h$ to be
the trivial homotopy, in which case $H$ is trivial and this reduces to the
usual definition of $\an{\alpha'',\alpha',\alpha}$ as the set of maps
$$ d(\mu(F,\alpha), \mu(\alpha'',G)) $$
where $F,G$ are nullhomotopies as in Definition \ref{def:toda}.
\end{rmk}

\begin{rmk}
The Toda bracket $\an{\alpha'',\alpha',\alpha}$ is a subset of the composition Toda bracket associated to the
composition
$S\too{\alpha''}X''\too{\alpha'}X'\too{\alpha}X$.
\end{rmk}

\begin{rmk}\label{rmk:compare-H}
	In Definition \ref{def:toda}, if two homotopies $H, H': S\cyl I \to X$ are homotopic relative to endpoints, then
	every element in the Toda bracket defined using $H$ is homotopic to an element in the
	Toda bracket defined using $H'$.
\end{rmk}

\subsection{Massey products}
\label{sec:massey}
\begin{definition}\label{def:massey}
Given $d_r$-cycles $\fra\in E_r(X),\fra'\in E_r(X'),\fra''\in E_r(X'')$ such that $\fra'\fra$ and $\fra''\fra'$
are boundaries,
define the Massey product
$$ \an{\fra'',\fra',\fra} = \{ \frb'\fra - \fra''\frb\st d_r(\frb')=\fra''\fra'\text{ and } d_r(\frb) = \fra'\fra \}. $$
\end{definition}

\begin{rmk}
	The reader may notice two subtle differences from the usual definition (e.g. see \cite[\S A.1.4]{green}) of the Massey product of classes in the $E_{r+1}$-page. First, we have defined $\fra$, $\fra'$, and $\fra''$ as cycles in the $E_r$-page.
	If the pairing on $E_r$ satisfies the Leibniz rule, then our Massey product clearly represents a class in $E_{r+1}$. Our results are only expected to be useful in cases where the Leibniz rule holds, but we do not need to make this assumption for the proofs.
Second, for the missing signs, recall (Remark \ref{rmk:signs}) that we are simplifying notation by assuming that all classes have degree zero.
\end{rmk}

\subsection{Setup}\label{sec:setup}
We now state our assumptions on the multiplicative structure of the filtrations and introduce notation for the elements in the various brackets. This material will be assumed for the rest of the paper. The reader who wishes to skip to the statements of the main theorems should read Assumption \ref{ass:X}, the notation at the beginning of Section \ref{sec:assumptions2}, and Assumption \ref{ass:crossing}.

\subsubsection{Assumptions on pairings}\label{sec:assumptions1}
Now we fix notation and assumptions that will be used in the rest of the paper.
Fix objects $X$, $X'$, and $X''$ in $\sC$ along with a filtration (as in
\eqref{eq:tower}) on each object. We view the filtered objects $X_*$, $X'_*$, $X''_*$ as belonging to the
graded category $\sC^\Z$, where $X_i = X_0$ if $i< 0$. (Here $\Z$ is a
discrete category (no non-identity morphisms), which means that the image of
$X_*$ in $\sC^\Z$ forgets the structure maps $i: X_s\to X_{s-1}$; the
compatibility with the structure maps is addressed below.) We endow $\sC^\Z$ with a
symmetric monoidal structure via Day convolution. More explicitly, given $X_*,
Y_*\in \sC^\Z$ define $(X\tensor Y)_q = \bigsqcup_{i+j = q} X_i \tensor Y_j$
where $\sqcup$ denotes the colimit in $\sC$.

Note that an object $X\in \sC$ can be
viewed as a graded object via the diagonal $d(X)\in \sC^\Z$, defined by $d(X)_s
= X$ for all $s$, and an $A_3$-structure on $(X'',X',X)$ naturally gives rise to
an $A_3$-structure on $(d(X''), d(X'), d(X))$.

The setup (especially \eqref{item:mu_r} and \eqref{item:H_1} below) would be much less complicated if we assumed that the diagrams
\begin{equation}\label{eq:strict-mu0} \xymatrix{
X'_*\tensor X_*\ar[r]^-{\mu_0}\ar[d]  & X_*\ar[d]
\\X'_{*-r}\tensor X_*\ar[r]^-{\mu_0} & X_{*-r}
}\end{equation}
(and similarly for the other pairings) commuted strictly. However, this need not
be the case in our desired application (the motivic slice spectral sequence and more generally, colocalization towers in the framework
of \cite{GRSO}), and instead we must choose assumptions that are weak enough to
be satisfied by colocalization towers, but strong enough to make the main
argument work. In particular, our assumptions are stronger than simply asking
the above diagram to commute up to homotopy.

\begin{custom}{Assumption}\label{ass:X}
Fix $A_3$-structures $(\mu,h)$ on $(X'',X',X)$ and $(\mu_0,h_*)$ on $(X''_*,X'_*,X_*)$. Moreover:

\begin{enumerate} 
\item \label{item:H} Assume there is a map of $A_3$-structures
$(X''_*,X'_*,X_*)\too{\iota} (d(X''), d(X'), d(X))$.
\item \label{item:mu_r}
Assume that there exist maps $\mu_{-r}, \mu_r, \mu^{-r}, \mu^r$ for
$r\geq 1$ making the following diagrams commute.
$$  \xymatrix{
X''_*\tensor X'_*\ar[r]^-{\mu_0}\ar[d] & X'_*\ar[d]
\\X''_*\tensor X'_{*-r}\ar[r]^-{\mu_{-r}}\ar[d] & X'_{*-r}\ar[d]
\\X''\tensor X'\ar[r]^-{\mu} & X'
}
\hspace{30pt}
\xymatrix{
X'_{*+r}\tensor X_*\ar[r]^-{\mu^r}\ar[r]\ar[d] & X_{*+r}\ar[d]
\\X'_*\tensor X_*\ar[r]^-{\mu_0}\ar[d] & X_*\ar[d]
\\X'_{*-r}\tensor X_*\ar[r]^-{\mu^{-r}}\ar[r]\ar[d] & X_{*-r}\ar[d]
\\X'\tensor X\ar[r]^-\mu & X
} \hspace{30pt}
\xymatrix{
X''_*\tensor X_{*+r}\ar[r]^-{\mu_r}\ar[r]\ar[d] & X_{*+r}\ar[d]^i
\\X''_*\tensor X_*\ar[r]^-{\mu_0}\ar[d] & X_*\ar[d]
\\X''\tensor X\ar[r]^-{\mu} & X
} $$
\item \label{item:H_1} Assume that there are homotopies $\til{H}_1$ from $\mu^r$ to $\mu_0$ and $\til{H}_2$ from $\mu_0$ to $\mu_r$ such that
$$ \xymatrix{
X'_*\tensor X_*\tensor I\ar[r]^-{\til{H}_1}\ar[d] & X_*\ar[dd]^\iota
\\X'_*\tensor X_*\ar[d] & 
\\X'\tensor X\ar[r]^-\mu & X
}\hspace{30pt}
\xymatrix{
X''_*\tensor X_*\tensor I\ar[r]^-{\til{H}_2}\ar[d] & X_*\ar[dd]^\iota
\\X''_*\tensor X_*\ar[d]
\\X''\tensor X\ar[r]^-{\mu} & X
}
$$
commute. (Note that $\mu^r$ and $\mu_r$ have been reindexed to be maps
$X'_*\tensor X_*\to X_*$ and $X''_*\tensor X_*\to X_*$, respectively.)
\item \label{item:mu_**} Assume there are maps $\mu_{*,1}$ making the following
diagrams commute.
$$ \xymatrix{
X'_*\tensor X_*\ar[r]^-{\mu_0}\ar[d] & X_*\ar[d]^-p
\\X'_{*,1}\tensor X_{*,1}\ar[r]^-{\mu_{*,1}} & X_{*,1}
} \hspace{30pt}
\xymatrix{
X''_*\tensor X_*\ar[r]^-{\mu_0}\ar[d] & X_*\ar[d]^-p
\\X''_{*,1}\tensor X_{*,1}\ar[r]^-{\mu_{*,1}} & X_{*,1}
} $$

\item \label{item:E_r-pairings} There are pairings of spectral sequences
\begin{align*}
E_r(X')\tensor E_r(X) & \to E_r(X)
\\E_r(X'')\tensor E_r(X) & \to E_r(X)
\end{align*}
induced by $\mu_{*,1}$ and $\mu_{*,1}$. (Recall that $E_1^{*,f}(X) = \pi_*(X_{f,1})$ and the $E_r$-page is a subquotient of the $E_1$-page.)
\end{enumerate}
\end{custom}

\begin{rmk}\label{rmk:filtered}
The synthetic perspective (\cite{GWX}, \cite{pstragowski}, \cite{GIKR}, \cite{burklund-moore}) views spectral sequences as filtered spectra, and
synthetic tools are well-equipped to address questions about monoids in filtered
spectra. Our setting is more general: the setting of modules over
monoids in filtered spectra corresponds to the special case when
\eqref{eq:strict-mu0} and analogous diagrams commute strictly.
In this special case, assumptions \eqref{item:mu_r} and \eqref{item:H_1} hold
trivially.
\end{rmk}

\begin{rmk}
Davis and Snaith \cite[Theorem 2.2]{davis-snaith} give an
abstract proof of the differential and extension theorem for Massey products in
spectral sequences. Their theorem has a hypothesis called the ``exchange
condition.'' If \eqref{eq:strict-mu0} commutes, one can show that our setup
implies the exchange condition.
\end{rmk}


\subsubsection{The homotopy classes $\alpha$, $\alpha'$, $\alpha''$}\label{sec:assumptions2}
Suppose there are maps 
\begin{align*}
{\alpha}: & S\to X & {\alpha}': & S\to X' & \alpha'': & S\to X'',
\end{align*}
with nullhomotopic products $\mu(\alpha',\alpha)\hteq 0\hteq \mu(\alpha'',\alpha')$.
Furthermore, suppose that ${\alpha}$, ${\alpha}'$, and ${\alpha}''$
are detected by nonzero permanent cycles 
\begin{align*}
{\fra} & \in E_r^{*,f-f'}(X) & {\fra}' & \in E_r^{*,f'-f''}(X') & {\fra}'' & \in E_r^{*,f''}(X''),
\end{align*}
respectively, such
that $\fra'\fra$ and $\fra''\fra'$ are $d_r$-boundaries.
Then there are maps 
\begin{align*}
\alpha_{**} & :S\to X_{f-f',1} & \alpha'_{**} & :S\to X'_{f'-f'',1} & \alpha''_{**} & :S\to X''_{f'',1}
\end{align*}
representing $\fra$, $\fra'$, and $\fra''$, and since these
are permanent cycles, they lift to 
\begin{align*}
\alpha_* & :S\to X_{f-f'} & \alpha'_* & :S\to X'_{f'-f''} & \alpha''_* & :S\to X''_{f''}
\end{align*}
satisfying
$\alpha=\iota\alpha_*$ and $\alpha_{**}=p\alpha_*$ (and similarly for $\alpha'$,
$\alpha''$) where $\iota:X_*\to X$ and $p:X_*\to X_{*,1}$ are the maps of graded
objects guaranteed in Assumption \ref{ass:X}\eqref{item:H},\eqref{item:mu_**}.

Strict compatibility of $\iota$ and $p$ with the multiplicative structures  (Assumption \ref{ass:X}\eqref{item:H},\eqref{item:mu_**}) implies that $\iota \mu_0(\alpha'_*,\alpha_*) = \mu(\alpha',\alpha)$ and $p\mu_0(\alpha'_*,\alpha_*) = \mu_{*,1}(\alpha'_{**},\alpha_{**})$,
and Assumption \ref{ass:X}\eqref{item:E_r-pairings} says that $\mu_{*,1}(\alpha'_{**},\alpha_{**}):S\to X_{f-f'',1}$ descends to the $E_r$ product $\fra'\fra$.

Beware that we need not have $\mu_0(\alpha''_*,\alpha'_*)=0=\mu_0(\alpha'_*,\alpha_*)$ or
$\mu_{*,1}(\alpha''_{**},\alpha'_{**})=0=\mu_{*,1}(\alpha'_{**},\alpha_{**})$; compare with Lemma \ref{lem:compositions-null} below.

\begin{custom}{Assumption}\label{ass:crossing}
In the rest of the paper,
assume that $\{E_*(X)\}$ satisfies the crossing differentials hypothesis
(Definition \ref{def:cross}) in degree $(r,|\fra'\fra|)$ and
$\{E_*(X')\}$ satisfies the crossing differentials hypothesis
in degree $(r,|\fra''\fra'|)$. Moreover, we assume that the spectral sequences
$\{ E_*(X) \}$ and $\{ E_*(X') \}$ are weakly convergent (see
Definition \ref{def:weakly-convergent}).
\end{custom}

\begin{lemma}\label{lem:compositions-null}
Given Assumption \ref{ass:crossing},
the following compositions are nullhomotopic.
$$ \xymatrix@C=43pt{
S\ar[r]^-{\mu_0(\alpha'_*,\alpha_*)} & X_{f-f''}\ar[r]^-i & X_{f-f''-r}
\\S\ar[r]^-{\mu_0'(\alpha''_*,\alpha'_*)} & X'_{f'}\ar[r]^-i & X'_{f'-r}
}\hspace{30pt}\xymatrix@C=43pt{
S\ar[r]^-{\mu_{*,1}(\alpha'_{**},\alpha_{**})} & X_{f-f'',1}\ar[r]^-i & X_{f-f''-r,1}\hspace{70pt}
\\S\ar[r]^-{\mu_{*,1}(\alpha''_{**},\alpha'_{**})} & X'_{f',1}\ar[r]^-i & X'_{f'-r,1}\hspace{70pt}
} $$
\end{lemma}
\begin{proof}
By Assumption \ref{ass:X}\eqref{item:H}, $\mu(\alpha',\alpha):S\to X$ factors as
$\iota \mu_0(\alpha'_*,\alpha_*)$. Since this is null and (using Assumption \ref{ass:X}\eqref{item:mu_**},\eqref{item:E_r-pairings})
$\mu_0(\alpha'_*,\alpha_*)$ lifts $\fra'\fra$ which is assumed to be a boundary, Proposition
\ref{prop:crossing-diff} implies that $i\mu_0(\alpha'_*,\alpha_*):S\to X_{f-f''-r}$ is
null. The corresponding graded statement follows from composing with
$p:X_{f-f''-r}\to X_{f-f''-r,1}$ and using Assumption \ref{ass:X}\eqref{item:mu_**}.
The remaining two statements are similar.
\end{proof}

\section{Comparison statements}\label{sec:comparison}
We assume the notation in Section \ref{sec:setup}.
\subsection{Comparing the Toda bracket and Massey product} \label{sec:comparison-Er}
The goal of this section is to prove Theorem \ref{thm:main},
which compares the Toda bracket $\an{\alpha'',\alpha',\alpha}$ with the Massey
product $\an{\fra'',\fra',\fra}$. The general idea is to define another nonempty bracket
$\an{\alpha''_*,i\alpha'_*,\alpha_*}$ and show that there are maps 
$$ \an{\fra'',\fra',\fra}\btoo{p} \an{\alpha''_*,i\alpha'_*, \alpha_*}\too{\iota} \an{\alpha'',\alpha',\alpha} $$
where $\iota: (X''_*,X'_*,X_*)$ is the map from Assumption
\ref{ass:X}\eqref{item:H} and $p:(X''_*,X'_*,X_*)\to (X''_{*,1}\to
X'_{*,1},X_{*,1})$ is from Assumption \ref{ass:X}\eqref{item:mu_**}. This would
be straightforward if the diagram \eqref{eq:strict-mu0} commuted strictly: given
nullhomotopies $F_*$ and $G_*$ of $\mu_0(\alpha''_*,i\alpha'_*)$ and
$\mu_0(i\alpha', \alpha_*)$, respectively, and a homotopy $H_*$ between
$\mu_0(\mu_0(\alpha''_*,\alpha'_*),\alpha_*)$ and
$\mu_0(\alpha''_*,\mu_0(\alpha'_*,\alpha_*))$ we would consider
\begin{equation}\label{eq:star-toda} d(\mu_0(F_*,\alpha_*), iH_*, \mu_0(\alpha''_*,G_*)) \in
\an{\alpha''_*,i\alpha_*, \alpha_*}. \end{equation}
However, with our current assumptions, this element is not well-defined, as
$\mu_0(F_*,\alpha_*):CS\to X_*$ restricted to $S \into CS$ is
$\mu_0(\mu_0(\alpha''_*,i\alpha'_*),\alpha_*)$, which is not an endpoint for
$iH_*$.
Instead, we must use the variety of multiplication maps on $(X''_*,X'_*,X_*)$
described in Assumption \ref{ass:X}\eqref{item:mu_r},\eqref{item:H_1} to write
down a homotopy between (variants of) these two elements, that restricts to the
trivial homotopy after applying $\iota$.

For ease of notation let
\begin{align*}
\alpha_{21} & = \mu_0(\alpha''_*,\alpha'_*)
\\\alpha_{10} & = \mu_0(\alpha'_*,\alpha_*).
\end{align*}
From Lemma \ref{lem:compositions-null} we have nullhomotopies
\begin{align*}
 & F_*:CS \to X'_{*-r} \text{  of }\  i\alpha_{21},  \text{ and}
\\ & G_*:CS\to X_{*-r}\text{ of }\ i\alpha_{10}.
\end{align*}

The setup ensures that the following diagrams commute strictly, except for the
dotted arrows.

$$ \uppercurveobject{{ }}
\twocellhead{{ }}
\xymatrix{
S\ar@/^1pc/[rrr]|-{(\alpha_{21},\, \alpha_*)}\ar[rr]_-{(\alpha''_*,\alpha'_*,\alpha_*)}\ar[d] && X''_*\tensor X'_*\tensor
X_*\ar[r]_-{(\mu_0,\,\Id)}\ar[d]^i & X'_*\tensor
X_*\ar[r]_-{\mu^r}\ar[d]^i\ruppertwocell^{}{^\hspace{28pt}H_1}\ar@{.>}@/^1.5pc/[r]^-{\mu_0} & X_*\ar[d]^i
\\ CS\ar@/_2pc/[rrr]_-{(F_*,\alpha_*)}&& X''_*\tensor X'_{*-r}\tensor X_*\ar[r]^-{(\mu_{-r},\,\Id)} & X'_{*-r}\tensor X_*\ar[r]^-{\mu_0}\ar[d]^p & X_{*-r}\ar[d]^p
\\ && & X'_{*-r,1}\tensor X_{*,1}\ar[r]^-{\mu_{*,1}} & X_{*-r,1}
}$$

$$ \uppercurveobject{{ }}
\twocellhead{{ }}
\xymatrix{
S\ar@/^1pc/[rrr]|-{(\alpha''_*,\, \alpha_{10})}\ar[rr]_-{(\alpha''_*,\alpha'_*,\alpha_*)}\ar[d] && X''_*\tensor X'_*\tensor
X_*\ar[r]_-{(\Id,\,\mu_0)}\ar[d]^i & X''_*\tensor X_*\ar[r]_-{\mu_r}\ar[d]^i\ar[d]^i\ruppertwocell{\ H_2}\ar@{.>}@/^1.5pc/[r]^-{\mu_0} & X_*\ar[d]^i
\\ CS\ar@/_2pc/[rrr]_-{(\alpha''_*,G_*)}&& X''_*\tensor X'_{*-r}\tensor
X_*\ar[r]^-{(\Id,\,\mu^{-r})} &
X''_*\tensor X_{*-r}\ar[r]^-{\mu_0}\ar[d]^p & X_{*-r}\ar[d]^p
\\ && & X''_{*,1}\tensor X_{*-r,1}\ar[r]^-{\mu_{*,1}} & X_{*-r,1}
}$$

Assumption \ref{ass:X}\eqref{item:H},\eqref{item:H_1} allows us to fix the following data:
\begin{itemize} 
\item a homotopy $H_*$ from $\mu_0(\mu_0(\alpha''_*,\alpha'_*),\alpha_*) =
\mu_0(\alpha_{21},\alpha_*)$ to
$\mu_0(\alpha''_*,\mu_0(\alpha'_*,\alpha_*)) = \mu_0(\alpha''_*,\alpha_{10})$
whose composition with $\iota:X_*\to X$ recovers the homotopy $H$ defined by the
$A_3$-structure on the triple $(X'',X',X)$;
\item a homotopy ${H}_1$
from $\mu^r(\alpha_{21},\alpha_*)$ to $\mu_0(\alpha_{21},\alpha_*)$ whose
composition with $\iota:X_*\to X$ is constant at
$\mu(\mu(\alpha'',\alpha'),\alpha)$;
\item a homotopy ${H}_2$ from $\mu_0(\alpha''_*,\alpha_{10})$ to
$\mu_r(\alpha''_*,\alpha_{10})$ whose composition with $\iota:X_*\to X$ is
constant at
$\mu(\alpha'',\mu(\alpha',\alpha))$.
\end{itemize}

Using all of this data we may define the crucial element
\begin{equation}\label{eq:gamma} \gamma = d(\mu_0(F_*,\alpha_*), i{H}_1, iH_*, i{H}_2,
\mu_0(\alpha''_*,G_*)): S \to X_{*-r}. \end{equation}
If \eqref{eq:strict-mu0} was required to commute, 
then ${H}_1$ and ${H}_2$ could be taken to be trivial, and $\gamma$
could be replaced by \eqref{eq:star-toda}. This strategy indeed works in
\S\ref{sec:E1}, where the Massey product is replaced by a Toda bracket of $E_1$ classes.

To state the comparison with the Massey product we need the following elements:
\begin{align*}
\beta^+ & = d(pF_*, j'C\alpha_{21})
\\\beta^- & = d(pG_*, jC\alpha_{10}),
\end{align*}
where
$j: CX_*\to X_{*-r,1}$ and $j':CX'_*\to X'_{*-r,1}$ are as in Lemma \ref{lem:source}.
That lemma implies that $\beta^+$ is the source of the
differential hitting $p\alpha_{21} = \mu_{*,1}(\alpha''_{**},\alpha'_{**})$ and $\beta^-$ is the source of the differential hitting $p\alpha_{10} = \mu_{*,1}(\alpha'_{**},\alpha_{**})$. Thus we have (using Assumption \ref{ass:X}\eqref{item:E_r-pairings}):
\begin{equation}\label{eq:massey} \mu_{*,1}(\beta^+,\alpha_{**}) -
\mu_{*,1}(\alpha''_{**}, \beta^-)\in \an{\fra'',\fra',\fra}.  \end{equation}
The next lemma identifies these two summands with terms that we can
more easily relate to $\gamma$.

\begin{lemma} \label{lem:massey-comparison1}
We have
\begin{align*}
\mu_{*,1}(\beta^+,\alpha_{**})  & \hteq d(\mu_{*,1}(pF_*,\alpha_{**}),
jC\mu^r(\alpha_{21},\alpha_*))
\\\mu_{*,1}(\alpha''_{**},\beta^-) & \hteq
d(\mu_{*,1}(\alpha''_{**},pG_*), jC\mu_r(\alpha''_*,\alpha_{10})).
\end{align*}
\end{lemma}
\begin{proof}
Applying Lemma \ref{lem:colimit} to $\mu_{*,1}(-,\alpha_{**}): X'_{*-r,1}\to
X_{*-r,1}$, we have
\begin{align*}
\mu_{*,1}(\beta^+,\alpha_{**}) & = \mu_{*,1}(d(pF_*, j'C\alpha_{21}),\alpha_{**})
\\ & = d(\mu_{*,1}(pF_*,\alpha_{**}), \mu_{*,1}(j'C\alpha_{21},\alpha_{**})).
\end{align*}
The commutative diagram
$$ \xymatrix@C=5pt{
 & CS\ar[rr]^-{C\alpha_{21}} &  & CX'_*\ar[rr]^-{C\mu^r(-,\alpha_*)}\ar@/^0.5pc/[dddl]_{j'} &  &
 CX_*\ar@/^0.5pc/[dddl]_j
\\S\ar[ru]\ar[rr]^-{\alpha_{21}} &&
X'_*\ar[d]\ar[rr]^-{\mu^r(-,\alpha_*)}\ar[ru] &&
X_*\ar[d]\ar[ru]
\\ && X'_{*-r}\ar[rr]^-{\mu_0(-,\alpha_*)}\ar[d]_p &  & X_{*-r}\ar[d]
\\ && X'_{*-r,1}\ar[rr]^-{\mu_{*,1}(-,\alpha_{**})} &  & X_{*-r,1} 
}$$
shows that $\mu_{*,1}(j'C\alpha_{21},\alpha_{**}) = jC\mu^r(\alpha_{21},
\alpha_*)$ as nullhomotopies $CS\to X_{*-r,1}$. The other statement is similar,
using the factorization $S \too{\alpha_{10}} X_* \ttoo{\mu_r(\alpha''_*,-)}
X_*$ of $\mu_r(\alpha''_*,\alpha_{10})$.
\end{proof}

Next we relate $\gamma$ to a sum involving terms that look like the right hand
side in Lemma \ref{lem:massey-comparison1}.
\begin{lemma} \label{lem:massey-comparison2}
We have
$$ p\circ \gamma \hteq d(p\mu_0(F_*,\alpha_*), jC\mu^r(\alpha_{21},\alpha_*)) -
d(p\mu_0(\alpha''_*,G_*), jC\mu_r(\alpha''_*,\alpha_{10})) $$
where $\gamma$ is defined in \eqref{eq:gamma}.
\end{lemma}
\begin{proof}
Let $f_1 = \mu^r(\alpha_{21},\alpha_*)$ and $g_1 =
\mu_r(\alpha''_*,\alpha_{10})$. We will show that 
\begin{equation}\label{eq:pgamma} p\circ \gamma \hteq d(p\mu_0(F_*,\alpha_*), jCf_1, 0, jCg_1,
p\mu_0(\alpha''_*,G_*)); \end{equation}
this suffices because it is homotopic to $d(p\mu_0(F_*,\alpha_*),jCf_1) +
d(jCg_1,p\mu_0(\alpha''_*,G_*))$, which is homotopic to the desired difference by
Lemma \ref{lem:difference}.
Recall that for any
$f:S\to X_*$ we have a nullhomotopy 
$$ CS\too{Cf} CX_*  \too{j} X_{*-r}\djunion{X_*} CX_* = X_{*-r,r} \to X_{*-r,1} $$
of $p\circ i\circ f$, and moreover this is functorial in $f$.
We represent the desired homotopy in \eqref{eq:pgamma} as follows, where the top row represents $p \circ \gamma$ and the bottom row represents the right hand side of
\eqref{eq:pgamma}.
For ease of notation we omit $j$ and $p:CX_{*-r}\to X_{*-r,1}$ from the notation in some places.

\begin{tikzpicture} 
\draw[-] (0,2) to node[above]{$p\mu_0(F_*,\alpha_*)$} (3,2) to
node[above]{$pi{H}_1$} (6,2) to
node[above]{$piH_*$} (9,2) to node[above]{$pi{H}_2$} (12,2) to
node[above]{$p\mu_0(\alpha''_*,G_*)$} (15,2);
\draw[-] (0,0) to node[below]{$p\mu_0(F_*,\alpha_*)$} (3,0) to node[below]{$Cf_1$} (6,0) to
node[below]{0} (9,0) to node[below]{${Cg_1}$} (12,0) to
node[below]{$p\mu_0(\alpha''_*,G_*)$} (15,0);
\coordinate[label={$=$}](A) at (1.5,0.8);
\coordinate[label={$=$}](A) at (13.5,0.8);
\coordinate[label={$Cf_1$}](A) at (3.75,0.4);
\coordinate[label={$C(-)$}](A) at (5.25,1.2);
\coordinate[label={$C(-)$}](A) at (7.5,0.8);
\coordinate[label={$C(-)$}](A) at (9.75,1.2);
\coordinate[label={$Cg_1$}](A) at (11.25,0.4);
\coordinate[label={\tiny $Cf_2$}](A) at (6.25,0.8);
\coordinate[label={\tiny $Cg_2$}](A) at (8.75,0.8);
\draw[-] (3,2) to node[sloped,above,inner sep=0pt]{\tiny $Cf_1$} (6,0);
\draw[-] (12,2) to node[sloped,above,inner sep=0pt]{\tiny $Cg_1$} (9,0);
\node[pin={[pin edge=<-, pin distance=10pt]90:{\small $f_1=\mu^r(\alpha_{21},\alpha_*)$}}] (v2) at (3,2)  {};
\node[pin={[pin edge=<-, pin distance=10pt]90:{\small $f_2=\mu_0(\alpha_{21},\alpha_*)$}}] (v2) at (6,2)  {};
\node[pin={[pin edge=<-, pin distance=10pt]90:{\small $g_2=\mu_0(\alpha''_*,\alpha_{10})$}}] (v3) at (9,2)  {};
\node[pin={[pin edge=<-, pin distance=10pt]90:{\small $g_1=\mu_r(\alpha''_*,\alpha_{10})$}}] (v3) at (12,2)  {};
\draw[-] (0,0) to (0,2);
\draw[-] (3,0) to (3,2);
\draw[-] (6,0) to (6,2);
\draw[-] (9,0) to (9,2);
\draw[-] (12,0) to (12,2);
\draw[-] (15,0) to (15,2);
\end{tikzpicture}
\end{proof}

\begin{prop} \label{prop:massey-comparison}
We have $p\circ \gamma \in \an{\fra'',\fra',\fra}$.
\end{prop}
\begin{proof}
Combine \eqref{eq:massey} with Lemmas \ref{lem:massey-comparison1} and \ref{lem:massey-comparison2}.
\end{proof}

Now that we have shown compatibility of $\gamma$ with the Massey product, it
remains to show compatibility with the Toda bracket.

\begin{prop} \label{prop:toda-comparison}
We have $\iota \circ \gamma \in \an{\alpha'',\alpha',\alpha}$, where
$\iota:X_*\to X$ is the map from Assumption \ref{ass:X}\eqref{item:H}.
\end{prop}
\begin{proof}
We have
\begin{align}
\notag \iota \circ \gamma &= d(\iota \mu_0(F_*,\alpha_*), \iota {H}_1, \iota H_*,
\notag \iota {H}_2, \mu_0(\alpha''_*,G_*))
\\\notag  & = d(\mu(\iota F, \alpha), \Id, \iota H_*, \Id, \mu(\alpha'',\iota G_*))
\\\label{eq:aaa-toda} & \hteq d(\mu(\iota F, \alpha), \iota H_*, \mu(\alpha'', \iota G_*)),
\end{align}
where the first equality is an application of Lemma \ref{lem:colimit} and the
second uses the assumptions on $\iota i{H}_1 = \iota {H}_1$ and $\iota
i{H}_2 = \iota {H}_2$ and Assumption \ref{ass:X}\eqref{item:H} for the
compatibility of $\mu$ with $\mu_0$. By Assumption \ref{ass:X}\eqref{item:H},
$\iota H_*$ is the homotopy in the $A_3$ structure on $(X'',X',X)$, and $\iota
F_*$ and $\iota G_*$ are nullhomotopies of $\mu(\alpha'',\alpha')$ and
$\mu(\alpha',\alpha)$, respectively. Thus \eqref{eq:aaa-toda} is an element of
the Toda bracket.
\end{proof}

Now we obtain the main theorem. The indexing differs slightly from the version in Theorem \ref{thm:main-intro}.
\begin{thm}\label{thm:main}
Fix $r\geq 1$.
Given the setup in sections \ref{sec:assumptions1} and \ref{sec:assumptions2}
(including Assumption \ref{ass:crossing}),
there exists an element of the
Massey product $\an{\fra'',\fra',\fra}$ that is a permanent cycle
converging to an element of $\an{\alpha'',\alpha',\alpha}$.
\end{thm}
\begin{proof}
Define $\gamma\in \pi_*(X_{f-r})$ as in \eqref{eq:gamma}. By Proposition
\ref{prop:massey-comparison}, the projection of $\gamma$ to the $E_1$-page is an
element of the Massey product $\an{\fra'',\fra',\fra}$. On the other hand, this element
detects $\iota \gamma$, which is in $\an{\alpha'',\alpha',\alpha}$ by
Proposition \ref{prop:toda-comparison}.
\end{proof}

\subsection{Toda brackets in the $E_1$-page}\label{sec:E1}
In the previous section, we worked with $E_r$-page classes $\fra,\fra',\fra''$ for
$r\geq 1$ such that $\fra'\fra=0=\fra''\fra'$ in $E_{r+1}(X)$. In this section, we consider $\alpha$, $\alpha'$, $\alpha''$ with $\mu(\alpha',\alpha)=0=\mu(\alpha'',\alpha')$ that are detected by $\fra,\fra',\fra''$ in $E_1^{f-f'}(X)$,
$E_1^{f'-f''}(X')$, and $E_1^{f''}(X'')$, respectively, such that $\fra'\fra=0=\fra''\fra'$ in $E_1(X)$.

In this case, we cannot define a corresponding Massey product. Instead, we recall that $E_1^f(X) = \pi_*(X_{f,1})$ by definition, and our setup allows us to consider Toda brackets of such homotopy elements.
More precisely, there is a Toda bracket $\an{\fra'',\fra',\fra}$ given by the classes
$S\too{\fra''} X''_{f'',1}$, $S\too{\fra'} X'_{f'-f'',1}$, and
$S\too{\fra}X_{f-f',1}$, where the associativity homotopy is given by composing the associativity homotopy $H_*$ used to define $\an{\alpha''_*,\alpha'_*,\alpha_*}$ with $p:X_*\to X_{*,1}$.
The goal of this section is to prove Theorem \ref{thm:E1}, which is the analogue of Theorem \ref{thm:main} where the Massey product is replaced by $\an{\fra'',\fra',\fra}$.
We still make Assumption \ref{ass:crossing} in degrees $(0, |\fra'\fra|)$ and $(0, |\fra''\fra'|)$; despite the motivation about crossing differentials, this does not require the definition of a $d_0$ differential.


Proposition \ref{prop:crossing-diff} does not make sense in this case; the replacement is Lemma \ref{lem:crossing-diff-E1}.
\begin{lemma}\label{lem:crossing-diff-E1}
Suppose $\zeta\in \pi_n(X_f)$ is a class whose images in both $\pi_n(X)$ and
$\pi_n(X_{f,1})$ are zero.
Suppose $\{E_*(X)\}$ is weakly convergent and satisfies the crossing differentials hypothesis in degree $(0, n,f)$. Then $\zeta = 0$.
\end{lemma}
\begin{proof}
If not, then there exists $k\leq f$ such that the image of $\zeta$ in $\pi_n(X_k)$
is nonzero but its image in $\pi_n(X_{k-1})$ is zero. This means that there
exists $z\in \pi_{n+1}(X_{k-1,1})$ such that $\kappa(z)$ lifts to
$\zeta\in \pi_n(X_f)$. Then $d_r(z)\neq 0$ for some $r$:
if not, then $\zeta$ lifts to $\pi_n(X_{f+r})$ for all $r$, and in the notation
of Definition \ref{def:weakly-convergent} we would have $i\zeta \in \ker(R_k\to
R_{k-1})$, contradicting weak convergence. Moreover, since $p(\zeta)\in \pi_n(X_{f,1})$ is zero, $\zeta$ lifts to $\pi_n(X_{f+1})$, and so $r > f-(k-1)$.
So $z$ is a class in $E_{f-k+2}^{n+1,k-1}$ that is not a permanent cycle,
contradicting the crossing differentials hypothesis.
\end{proof}

\begin{prop} \label{prop:thm-part1-E1}
Given Assumption \ref{ass:crossing} with $r=0$, 
$\mu_0(\alpha'_*,\alpha_*):S\to X_{f-f'}$ and $\mu_0(\alpha''_*,\alpha'_*):S\to X'_{f'}$
are null. For every element $\gamma_*\in \an{\alpha''_*,\alpha'_*,\alpha_*}$, the composition $S\too{\gamma_*} X_f\too{\iota} X$
is an element of
$\an{\alpha'',\alpha',\alpha}$.
\end{prop}
\begin{proof}
The first sentence follows by Lemma \ref{lem:crossing-diff-E1}. Let $F_*$ and
$G_*$ be nullhomotopies of these two classes, and let $H_*$ be a homotopy
between $\mu_0(\mu_0(\alpha''_*,\alpha'_*),\alpha_*)$ and
$\mu_0(\alpha''_*,\mu_0(\alpha'_*,\alpha_*))$. Let
$$ \gamma_* = d(\mu_0(F_*,\alpha_*), H_*, \mu_0(\alpha''_*,G_*)). $$
By Assumption \ref{ass:X}\eqref{item:H}, we have
$\iota \circ \gamma_* = d(\mu(\iota F_*, \alpha), \iota H_*, \mu(\alpha'',\iota G_*))$,
which is an element of the Toda bracket $\an{\alpha'',\alpha',\alpha}$.
\end{proof}

\begin{prop}\label{prop:thm-part2-E1}
Given Assumption \ref{ass:crossing} with $r=0$, every class in the Toda bracket
$\an{\alpha''_*,\alpha'_*,\alpha_*}$ projects to a class in the Toda bracket
$\an{\fra'',\fra',\fra}$.
\end{prop}
\begin{proof}
This holds using the same straightforward argument as Proposition
\ref{prop:thm-part1-E1}, along with Assumption \ref{ass:X}\eqref{item:mu_**}.
\end{proof}

\begin{thm}\label{thm:E1}
Suppose $\mu(\alpha',\alpha)\hteq 0 \hteq \mu(\alpha'',\alpha')$ and
$\mu_{*,1}(\fra',\fra) \hteq 0 \hteq \mu_{*,1}(\fra'',\fra)$ where $\fra\in\pi_*(X_{f-f',1})$ is an $E_1$ representative of $\alpha\in\pi_*(X)$, and similarly for $\fra'$ and $\fra''$.
Given Assumption \ref{ass:crossing} with $r=0$,
there is a permanent cycle in the Toda bracket $\an{\fra'',\fra',\fra}$ converging to an element of
$\an{\alpha'',\alpha',\alpha}$.
\end{thm}
\begin{proof}
Combine Propositions \ref{prop:thm-part1-E1} and \ref{prop:thm-part2-E1} as in the proof of Theorem \ref{thm:main}.
\end{proof}

\section{Applications}\label{sec:examples}
\subsection{Strict compatibility of $i$ and $\mu$}
For many important examples, the diagram \eqref{eq:strict-mu0} is strictly
commutative, which implies Assumption
\ref{ass:X}\eqref{item:mu_r},\eqref{item:H_1} by setting $\til{H}_1$ and
$\til{H}_2$ to be constant, and $\mu_0=\mu_{-r}=\mu_r=\mu^{-r}=\mu^r$.

\begin{eg} \label{example:Adams}
Let $E$ be a commutative, associative ring object in $\sC$ and
consider the $E$-based Adams spectral sequence for computing the homotopy of an
$A_3$-ring object $R$ of $\sC$. This arises from the tower
$$ \xymatrix{
R\ar[d] & \bar{E}\tensor R\ar[l]\ar[d] & \bar{E}^{\tensor 2}\tensor R\ar[l]\ar[d] & \dots\ar[l]
\\E\tensor R & E\tensor \bar{E}\tensor R & E\tensor \bar{E}^{\tensor 2}\tensor R
}$$
where $\bar{E}$ is the homotopy fiber of the unit map $S\to E$. Write $X_s =
\bar{E}^{\tensor s}\tensor R$. Then the multiplication maps on the filtered
object
$$ \bar{E}^{\tensor s}\tensor R\tensor \bar{E}^{\tensor t}\tensor R \to
\bar{E}^{\tensor s}\tensor \bar{E}^{\tensor t}\tensor R\tensor R\to
\bar{E}^{\tensor s+t}\tensor R $$
arise from commuting $R$ past $\bar{E}^{\tensor t}$ and using the multiplication
on $R$. 
It is straightforward to see that the diagram \eqref{eq:strict-mu0} is strictly
commutative, and that Assumptions \ref{ass:X}\eqref{item:H},\eqref{item:mu_**}
hold. The induced pairing of $E_r$-pages (Assumption \ref{ass:X}\eqref{item:E_r-pairings})
is standard; see e.g. \cite[Theorem 2.3.3]{green}.
\end{eg}

\begin{eg}
The May spectral sequence 
$$ E_1 = \F_2[h_{i,j}\st i>0, j\geq 0] \implies \Ext_A^{*,*}(\F_2,\F_2) $$
where $A$ is the dual Steenrod algebra arises from a filtration of the cobar complex
$A^{\tensor *}$ defined by the grading $|\xi_i^{2^j}|=2i-1$ (see \cite[proof of Theorem 3.2.3]{green}). The cobar complex is strictly associative, and the structure maps in the filtration arise from inclusion of subcomplexes, and so \eqref{eq:strict-mu0} commutes.
The other parts of Assumption \ref{ass:X} are straightforward, and Theorem \ref{thm:main} recovers the May Convergence Theorem \cite[Theorem 4.1]{may-matric}.
\end{eg}

\subsection{Colocalization towers} \label{sec:colocalization}
Our setting is designed to fit with the theorems proved by
Gutierrez, R\"ondigs, Spitzweck, and {\O}stv{\ae}r \cite{GRSO} on
multiplicativity of colocalization towers.

Suppose $\sM$ is a stable combinatorial simplicial symmetric monoidal proper
model category with cofibrant unit. Let
$$ \dots \supseteq \cC_{-1} \supseteq \cC_0 \supseteq \cC_1 \supseteq \cC_2 \supseteq \dots $$
be a family
of full subcategories of $\Ho(\sM)$ satisfying the axioms in \cite[\S2.1]{GRSO}.
In particular, $\cC_i$ is the homologically non-negative part of some
$t$-structure on $\Ho(\sM)$ (see \cite[\S2.1]{GRSO}), and so
there are corresponding colocalization and localization
functors $c_i$ and $l_i$. By \cite[Lemma 2.1]{GRSO} these functors can be lifted
to colocalization and localization functors on $\mathcal{M}$.
For example (\cite[\S3.4]{GRSO}), if $\cC$ has an accessible $t$-structure where
$\cC_{\geq 0}$ is a symmetric monoidal subcategory, then the collection
$\{\cC_{\geq 0}[i] \st i\in \Z\}$ satisfies the axioms.

Our goal is to prove the following.

\begin{prop}\label{prop:GRSO-assumptions}
Let $X',X''$ be $A_3$-rings with an $A_3$-map $X''\to
X'$, and let $X$ be an $A_3$-$X'$-module. 
Assume the existence of colocalization and localization functors $c_i,l_i$ as
above, and consider the filtrations on $X''$, $X'$, and $X$ defined by $X''_i =
c_iX''$ and so on. Then $(X''_*,X'_*,X_*)$ can be given a multiplicative
structure satisfying Assumption \ref{ass:X}\eqref{item:H}--\eqref{item:mu_**}.
\end{prop}

We use the machinery of colored operads; see e.g. \cite[Appendix A]{GRSO} for an
introduction. Let $\mathcal{K}$ denote the Stasheff $A_3$-operad. We encode
the structure on the triple $(X'',X',X)$ in Proposition
\ref{prop:GRSO-assumptions} as an algebra structure over the following operad.
Define the colored operad $\Kmod$ with colors $\{ x'',x',x \}$ as follows:
\begin{align*}
\Kmod(t_1,\dots,t_n;t) & = \begin{cases}
\mathcal{K}(n) & \text{ if $t=x''$ and $t_i =x''$ for all $i$}
\\\mathcal{K}(n) & \text{ if $t=x'$ and $t_i \in \{ x',x'' \}$ for all $i$}
\\\mathcal{K}(n) & \text{ if $t=x$ and exactly one $t_i$ equals $x$}
\\0 & \text{ otherwise.}
\end{cases}
\end{align*}
In particular, the map $X''\to X'$ comes from the canonical element in
$\Kmod(x'';x')=\mathcal{K}(1)$, and the homotopy between $\mu(\mu(-,-),-)$ and
$\mu(-,\mu(-,-))$ comes from $\Kmod(x'',x',x;x) = \mathcal{K}(3)=\Delta^1$.
Then triples $(X'',X',X)$ satisfying the hypotheses of Proposition
\ref{prop:GRSO-assumptions} are $\Kmod$-algebras.


The strategy in \cite[\S2.3]{GRSO} for studying graded objects is to apply the
general setup to the categories $\cD_r = \prod_{i\in \Z}\cC_{i+r} \subseteq
\cC^\Z$. Let $d:\cC\to \cC^\Z$ denote the diagonal functor. There are corresponding colocalization and localization functors $c'_r$ and $l'_r$
where $c'_rd(X)_n = c_{r+n}X$. We consider the graded object $X_*$ where $X_s = c_s X$, and
write $X_{*+r}:=c'_rd(X)$.
Note that $(d(X''),d(X'),d(X))$ is an $\Kmod$-algebra in $\cC^{\Z}$. For ease of
notation, we denote this as simply $(X'',X',X)$.

The next technical lemma contains the main input we need from \cite{GRSO}. It
requires the following notation.
Given a colored operad $\OO$ and compatibly indexed collections $\mathbf{X},\mathbf{Y}$, let
$\Hom_{\OO}(\mathbf{X},\mathbf{Y})$ be the colored collection with
$$\Hom_{\OO}(\mathbf{X},\mathbf{Y})(t_1,\dots,t_n;t) = \begin{cases}
\Hom(\mathbf{X}(t_1)\tensor \dots \tensor \mathbf{X}(t_n);\mathbf{Y}(t)) &
\text{ if }\OO(t_1,\dots,t_n;t)\neq 0\\0 & \text{ otherwise.} \end{cases}$$

\begin{lemma} [{\cite[Theorem 4.1 and proof of Proposition 5.14]{GRSO}}]
\label{lem:GRSO-4.1}
Suppose $\OO$ is a $\mathcal{J}$-colored operad, and for
every color $t\in \mathcal{J}$ there is a colocalization $c_{s(t)}$ (with $s(t)\in \Z$) such that
$\OO(t_1,\dots,t_n;t)=0$ if $s(t_1)+ \dots + s(t_n) < s(t)$.
Given $\mathbf{X}\in \mathcal{C}^{\mathcal{J}}$, define
$Q\mathbf{X}(t) = c_{s(t)} \mathbf{X}(t)$, and
define the $\mathcal{J}$-colored collection $\mathcal{P}$ so
$\mathcal{P}(t_1,\dots,t_n;t)$ is the space of commutative diagrams
$$ \xymatrix{
Q\mathbf{X}(t_1)\tensor \dots \tensor Q\mathbf{X}(t_n)\ar[r]\ar[d] & Q\mathbf{X}(t)\ar[d]
\\\textbf{X}(t_1)\tensor \dots \tensor \textbf{X}(t_n)\ar[r] & \mathbf{X}(t).
}$$
Then the natural map $\tau:\mathcal{P}\to \Hom_{\OO}(\mathbf{X},\mathbf{X})$ is a trivial fibration
in the semi model structure from \cite[\S A.2]{GRSO}. Let $\OO_\infty$ be a
cofibrant resolution for $\OO$.
If $\mathbf{X}$ is an $\OO$-algebra, then there is an $\OO_\infty$-algebra structure $\psi$ on $Q\mathbf{X}$
that is compatible with the $\OO$-algebra structure on $\mathbf{X}$ in the
sense that
there is a commutative diagram of colored collections
$$ \xymatrix{
 &  & \mathcal{P}\ar[d]^\tau
\\ \Oiy\ar[r]\ar@{.>}[rru]^-\psi & \OO\ar[r] & \Hom_{\OO}(\mathbf{X},\mathbf{X}).
}$$
\end{lemma}

\begin{lemma}\label{item:GRSO-2}
Given the setup above, there is an $A_3$-structure on $(X''_*,X'_*,X_*)$
satisfying Assumption \ref{ass:X}\eqref{item:H},\eqref{item:mu_r}.
\end{lemma}

\begin{proof}
Reindex Assumption \ref{ass:X}\eqref{item:mu_r} so we are looking for the
following  diagrams along with homotopy associativity for the two
multiplications $X''_*\tensor X'_{*+r}\tensor X_*\to X_{*+r}$.
$$  \xymatrix{
X''_*\tensor X'_{*+r}\ar[r]^-{\mu_0}\ar[d] & X'_{*+r}\ar[d]
\\X''_*\tensor X'_*\ar[r]^-{\mu_{-r}}\ar[d] & X'_*\ar[d]
\\X''\tensor X'\ar[r]^-{\mu} & X'
}
\hspace{30pt}
\xymatrix{
X'_{*+2r}\tensor X_*\ar[r]^-{\mu^r}\ar[r]\ar[d] & X_{*+2r}\ar[d]
\\X'_{*+r}\tensor X_*\ar[r]^-{\mu_0}\ar[d] & X_{*+r}\ar[d]
\\X'_*\tensor X_*\ar[r]^-{\mu^{-r}}\ar[r]\ar[d] & X_*\ar[d]
\\X'\tensor X\ar[r]^-\mu & X
} \hspace{30pt}
\xymatrix{
X''_*\tensor X_{*+2r}\ar[r]^-{\mu_r}\ar[r]\ar[d] & X_{*+2r}\ar[d]^i
\\X''_*\tensor X_{*+r}\ar[r]^-{\mu_0}\ar[d] & X_{*+r}\ar[d]
\\X''\tensor X\ar[r]^-{\mu} & X
} $$
The strategy is as follows:
the reindexing enables us to first construct $\mu_{-r}$ and $\mu^{-r}$ by constructing a $\Kmod$-structure on $(X''_*,X'_*,X_*)$. Then, we construct $\mu_0$ compatible with this by using a new operad $\OO$ inspired by the morphism operad $\Mor_{\Kmod}$ of \cite[A.1.2]{GRSO} to obtain the appropriate multiplicative structure on the colocalization map $(X''_*,X'_{*+r},X_*)\to (X''_*,X'_*,X_*)$.

We will focus on the middle diagram, as
the other two diagrams are analogous. Moreover, we only need to construct the
bottom two squares; for the top square, $X'_{*+2r}\tensor X_*\to X'_{*+r}\tensor
X_*\too{\mu_0} X_{*+r}$ factors through $X_{*+2r}$ by the universal property of the
colocalization. (Note that we cannot simply construct
$\mu^{-r}$ and $\mu_0$ in this way, because we also need to show $\mu_0$ is
homotopy-associative after appropriate indexing; this will come from the
subsequent operad arguments.)
By Lemma \ref{lem:GRSO-4.1} with $s(t)=0$ for all colors $t$, we obtain a $\Kmod_\infty$-algebra structure $\psi_1$ on
$(X''_*,X'_*,X_*)$ compatible with $(X'',X',X)$. Choose $\bar{\mu}^{-r}\in
\Kmod_\infty(x',x;x)$ and let $\mu^{-r}=\psi_1(\bar{\mu}^{-r})$.

To construct $\mu_0$ compatible with $\mu^{-r}$,
let $\OO$ be the colored operad with color set
$$ \mathcal{J} = \{ x''_0,x'_{2r},x'_r,x'_0, x_{2r},x_r,x_0 \} $$
and 
$$ \OO(t_1,\dots,t_n;t) = \begin{cases} \Kmod_\infty(t_1,\dots,t_n;t) &
\text{ if } \sum_i sub(t_i) \geq sub(t)
\\0 & \text{ otherwise}
\end{cases} $$
where $sub(x_r)=r$, $sub(x_0)=0$, and so on.
An $\OO$-algebra is a tuple $(A_0,B_{2r},B_r,B_0,M_{2r},M_r,M_0)$
with contractible spaces of structure maps such as $M_r\to M_0$ and $B_r\tensor M_0\to M_r$.

Since the structure maps for $\OO$ come from $\Kmod_\infty$, the $\Kmod_\infty$-algebra structure
on $(X''_*,X'_*,X_*)$ can be repackaged as an $\OO$-algebra structure on
$\mathbf{X}_*:=(X''_*, X'_*,X'_*,X'_*,X_*,X_*,X_*)$.
Lemma \ref{lem:GRSO-4.1} applies where $s$ is the subscript function $sub$, and
gives an $\OO_\infty$-algebra structure $\psi$ on
$$ Q\mathbf{X}_* := (X''_*,X'_{*+2r},X'_{*+r},X'_*,X_{*+2r},X_{*+r},X_*). $$
This parametrizes a family of multiplications $X'_{*+r}\tensor X_*\to
X_{*+r}$, and we need to check that there exists one making the following diagram commute.
\begin{equation}\label{eq:middle-square} \xymatrix{
X'_{*+r}\tensor X_*\ar[r]^-{\mu_0}\ar[d] & X_{*+r}\ar[d]
\\X'_*\tensor X_*\ar[r]^-{\mu^{-r}}\ar[r] & X_*
}\end{equation}
This amounts to showing that $X'_{*+r}\tensor X_* \to X'_*\tensor X_*\too{\mu^{-r}}
X_*$ factors through $X_{*+r}$, where $\mu^{-r} = \psi_1(\bar{\mu}^{-r})$. Since
$\OO_\infty(x'_0,x_0;x_0)\to \OO(x'_0,x_0;x_0) = \Kmod_\infty(x',x;x)$ is a trivial fibration and the
point is cofibrant, we can lift $\bar{\mu}^{-r}$ to $\til{\mu}^{-r}\in
\OO_\infty(x',x_0;x_0)$ with $\psi(\til{\mu}^{-r}) = \mu^{-r}$.
One can find $a\in \Oiy(x_r;x_0)$ and $a'\in \Oiy(x'_r;x'_0)$ such that $\psi(a)$ and $\psi(a')$
are the colocalization maps $X_{*+r}\to X_*$ and $X'_{*+r}\to X'_*$, since $\OO_\infty\to \OO$ is a levelwise trivial fibration and this is true for $\OO$.
$$ 
\xymatrix{
(\til{\mu}^{-r},a')\ar@{}[r]|-\in & \Oiy(x'_0,x_0;x_0)\tensor
\Oiy(x'_r;x'_0)\ar[r]^-\psi\ar[d] &
 \mathcal{P}(x'_0,x_0;x_0)\tensor \mathcal{P}(x'_r;x'_0)\ar[d]
\\ & \Oiy(x'_r,x_0;x_0)\ar[r]^-\psi & \mathcal{P}(x'_r,x_0;x_0) \ni f
} $$
$$ \xymatrix{
 & \{ a \}\tensor \Oiy(x'_r,x_0; x_r)\ar[r]^-\psi\ar[d]^\gamma &
\mathcal{P}(x_r;x_0)\tensor \mathcal{P}(x'_r,x_0; x_r)\ar[d]
\\\ast\ar[r]\ar@{.>}[ru] &  \Oiy(x'_r,x_0;x_0)\ar[r]^-\psi & \mathcal{P}(x'_r,x_0;x_0)\ni f
} $$
Again using the fact that $\OO_\infty\to \OO$ is a levelwise trivial fibration, it is easy to check that $\gamma$ is a trivial fibration, and hence the lift exists.
Define $f\in \mathcal{P}(x'_r,x_0;x_0)$ as the image of $(\til{\mu}^{-r},a')$
under the first diagram; this represents the bottom left composition in
\eqref{eq:middle-square}.
This determines the bottom row of the second diagram, and if
$\til{\mu}_0\in \Oiy(x'_r,x_0;x_r)$ is the image of the lift in the second
diagram, then we take $\mu_0 = \psi(\til{\mu}_0)$. The other maps 
$\mu_0:X''_*\tensor X'_{*+r}\to X'_{*+r}$
and $\mu_0:X''_*\tensor X_{*+r}\to X_{*+r}$ can be defined similarly, and the associativity follows from the $\OO_\infty$-algebra structure.
\end{proof}

\begin{lemma}\label{lem:mu-unique}
Given $\mu:X'\tensor X\to X$ and
a pair of maps $\mu_0,\mu_1: X'_*\tensor X_*\to X_*$ making the following
diagram in $\cC^\Z$ commute for $i=0,1$,
\begin{equation}\label{eq:mu_t-htpy} \xymatrix{
X'_*\tensor X_*\ar[r]^-{\mu_i}\ar[d] & X_*\ar[d]
\\X'\tensor X\ar[r]^-\mu & X
}\end{equation}
there is a homotopy $\mu_*: I\to \Hom(X'_*\tensor X_*,X_*)$
between them such that \eqref{eq:mu_t-htpy} commutes for $i\in I$.
\end{lemma}
\begin{proof}
Let $\mathcal{P}$ be as in Lemma
\ref{lem:GRSO-4.1} with $\mathbf{X} = (X'',X',X)$, the operad $\Kmod$, and the
colocalization $s(t)=0$ for all colors $t$. The lemma says that $\mathcal{P}(x';x)\to
\Hom_{\Kmod}(\mathbf{X},\mathbf{X})(x',x;x) = \Hom(X'\tensor X,X)$ is a trivial fibration,
and since all simplicial sets are cofibrant, we have a lifting diagram as
follows.
$$ \xymatrix{
 & \mathcal{P}(x',x;x)\ar[d]
\\\ast\ar[r]_-\mu\ar@{.>}[ru] & \Hom(X'\tensor X,X)
}$$
Then $\mu_0$ and $\mu_1$ represent two lifts, and 
it is a standard model category fact that any two lifts are homotopic via a homotopy
over $\Hom(X'\tensor X,X)$.
\end{proof}


\begin{proof} [Proof of Proposition \ref{prop:GRSO-assumptions}]
Parts \eqref{item:H} and \eqref{item:mu_r} are proved as Lemma
\ref{item:GRSO-2}. Part \eqref{item:H_1} follows from this using Lemma
\ref{lem:mu-unique} and reindexing $\mu^r$ as a map $X'_*\tensor X_*\to X_*$ and
similarly for $\mu_r$.

We now turn to Part \eqref{item:mu_**}.
Following the definition of the full slice $s_*$ in \cite[\S5]{GRSO}, consider
$(l'_1c'_0 d(X''), l'_1c'_0d(X'), l'_1 c'_0d(X))$. A straightforward
application of \cite[Theorem 4.2]{GRSO} (similar to the proof of \cite[Theorem
5.16]{GRSO}) shows that this object is an
$\Kmod$-algebra, with strict comparison to $(d(X''),d(X'), d(X))$.
\end{proof}

The equivariant regular slice spectral sequence is an example of a tower of
localizations; Moss' theorem has been proved in this context by Ullmann
\cite{ullman}.
The main example in \cite{GRSO} is the motivic slice spectral sequence, which we
discuss next; Moss' theorem in this context is new.

\subsection{The $\R$-motivic effective slice spectral sequence} \label{sec:R-motivic}
Our motivation for writing this paper is to study Toda brackets in the
$\R$-motivic effective slice spectral sequence for the sphere at $p=2$. This is the
spectral sequence associated to Voevodsky's effective slice filtration: given an
effective motivic spectrum $X$, there is a tower
$$ 
\adjustbox{scale=1}{\xymatrix{
\hspace{-20pt}X = f_0 X\ar[d] & f_1X\ar[l]\ar[d] & f_2X\ar[l] & \cdots
\\s_0X & s_1X & \cdots
}}$$
which gives rise to the effective slice spectral sequence.
By \cite{Pelaez}, if $X$ is a ring spectrum, we get a multiplicative
pairing of spectral sequences $E_r(X)\tensor E_r(X)\to E_r(X)$.
The quotients $s_qX$ are called \emph{slices}. See
\cite{levine-convergence,RSO19} for more details.

\begin{lemma}
Let $X'',X'$ be motivic (strict) ring spectra with a ring map $X''\to X'$, and let $X$ be
an $X'$-module.
There are multiplications on the triples $(X'',X',X)$, $(f_*(X''), f_*(X'),
f_*(X))$, and
$(s_*(X''), s_*(X'), s_*(X))$ satisfying Assumption \ref{ass:X}.
\end{lemma}
For example, we are particularly interested in the case where $X''=X'=X$ are the
sphere spectrum $S$, which is strictly commutative.
\begin{proof}
By \cite[\S6]{GRSO}, the setup of \cite[\S2.1]{GRSO} applies to the effective slice filtration on the category $\M:=\Sp^\Sigma_T(k)$ of symmetric motivic spectra,
and so Assumption \ref{ass:X}\eqref{item:H}--\eqref{item:mu_**} follow from
Proposition \ref{prop:GRSO-assumptions}.
For Assumption \ref{ass:X}\eqref{item:E_r-pairings}, Pelaez \cite[Theorem 3.6.16]{Pelaez} shows that there is a natural multiplication on the $E_r$-page, defined using the usual definition of $E_r$ as a subquotient of $E_1$. 
\end{proof}

R{\"o}ndigs, Spitzweck, and {\O}stv{\ae}r \cite{RSO19} calculated the slices of
the motivic sphere spectrum:
\begin{equation}\label{eq:RSO} s_*(S) = H\Z\tensor \Ext_{MU_*MU}^{*,*}(MU_*, MU_*)
\end{equation}
where the notation means that a copy of $\Z/2^n$ in the $\Ext$ term corresponds
to an (appropriately degree-shifted) $H\Z/2^n$ summand, where $H\Z/2^n$ is the motivic Eilenberg--Maclane
spectrum. In the identification \eqref{eq:RSO},
(higher) multiplicative data about the Adams--Novikov $E_2$-term
$\Ext^{*,*}_{MU_*MU}(MU_*, MU_*)$ carries over to the slice $E_1$-page
$\pi_{*,*}(s_*(S))$. In practice, we work one prime at a time, and focus on the
$p=2$ case.
Combining results of \cite{RSO19} and
Hu, Kriz, and Ormsby \cite{HKO11}, the 2-completed slice spectral sequence for
the $\R$-motivic sphere converges to the homotopy groups of the 2-completed
sphere, which we denote by $\pi_{*,*}^\R(S)$.

We illustrate how one uses the Toda bracket convergence theorem to compute
relations and permanent cycles in the $\R$-motivic slice spectral sequence.

\begin{eg}\label{ex:nu}
There is a strictly defined Adams--Novikov Massey product
$\an{\alpha_1,2,\alpha_1} = 2\alpha_{2/2}$. The class $\alpha_1$ is a permanent
cycle converging to $\eta\in \pi_{1,1}^\R(S)$ and $\alpha_{2/2}$ is a permanent
cycle converging to $\nu\in \pi_{3,2}^\R(S)$.
We have $2\cdot \alpha_1=0$ in the
slice $E_1$-page but $2\cdot \eta$ is nonzero in
$\pi_{*,*}^\R(S)$. However,
there is a class $\omega= 2+\rho\eta$ detected by 2 in
the slice spectral sequence (as $\rho\eta$ has higher slice filtration) such that $\omega\eta=0 = \omega\rho$.
Hence there is a Toda bracket
$\an{\alpha_1,2,\alpha_1}$ in the slice $E_1$-page 
containing $2\alpha_{2/2}$ and there is a corresponding Toda bracket
$\an{\eta,\omega,\eta}$ in homotopy that is also strictly defined.
Then Theorem \ref{thm:E1} implies that
$2\alpha_{2/2}$ is a permanent cycle in the $\R$-motivic slice spectral
sequence converging to $\omega\nu$ which is contained in the homotopy Toda
bracket $\an{\eta,\omega,\eta}$.


While this example is in low enough stems that the Toda bracket argument is not
needed, this argument applies to a wealth of other Adams--Novikov Massey
products.
\end{eg}

\begin{eg}\label{ex:ta1}
In $\pi_{*,*}^\R (H\Z/2^n)$, there is a strictly defined Toda bracket
$\an{\rho,2,2^{n-1}} = 2^{n-1}\tau$, which can be seen using the
spectral sequence associated to the cofiber sequence of motivic
Eilenberg-Maclane spectra
$$ H\Z/2\to H\Z/2^n\to H\Z/2^{n-1}. $$
This is another source of slice $E_1$-page Toda brackets: if $x\in
\Ext_{MU_*MU}^{*,*}(MU_*, MU_*)$ generates a copy of $\Z/2^n$, then
one may consider the $E_1$-page Toda bracket $\an{\rho, 2, 2^{n-1}x} \ni 
2^{n-1}\tau x$. For example, there is a strictly defined $E_1$-page Toda bracket
$\an{\rho,2,\alpha_1} = \tau \alpha_1$. Using the notation in Example \ref{ex:nu}, Theorem \ref{thm:E1} implies that the $\pi_{*,*}^\R(S)$ Toda bracket
$\an{\rho,\omega,\eta}$ contains $\tau \eta$.
\end{eg}

\begin{eg}
There are also higher slice Massey products which converge to Toda brackets. For
example, the slice differential $d_1(\tau^2) = \rho^2\tau \alpha_1$ implies that
$\an{2, \rho^2,\tau \alpha_1}$ is strictly defined as a slice $E_2$-page Massey product
consisting of $2\tau^2$. The products $\omega\cdot \rho^2$ and $\rho^2\cdot \tau
\eta$ are zero in homotopy (see Example \ref{ex:nu}), and so the Toda bracket
$\an{\omega, \rho^2, \tau \eta}$ in $\pi_{*,*}^\R(S)$ is defined and contains
$\omega\tau^2$. Then we may
deduce the hidden multiplication $\eta\cdot \omega\tau^2$ using a shuffle
argument as follows. We have
$$ \eta \cdot \an{\omega,\rho^2,\tau \eta} = \an{\eta,\omega,\rho^2} \tau
\eta. $$
By Example \ref{ex:ta1}, the right hand side contains $\rho(\tau \eta)^2$.
Moreover, the right hand side is strictly defined: $\an{\eta,\omega,\rho^2}$ contains
indeterminacy generated by $\rho^3\nu$, but this in annihilated by
$\tau \eta$, since $\rho^2 \tau \eta = 0$. Thus we have
$$ \eta \cdot \omega\tau^2 = \rho (\tau \eta)^2. $$
\end{eg}

\printbibliography

\end{document}